\documentclass[11pt]{amsart}
\usepackage{amsmath,amssymb,mathrsfs,color}
\usepackage{hyperref}
\hypersetup{hypertex=true,
            colorlinks=true,
            linkcolor=blue,
            anchorcolor=blue,
            citecolor=blue}

\let\cal=\mathcal
\def\N{{\mathbb N}}

\def\R{{\mathbb R}}

\newcommand{\rmd}{{\rm d}}
\newtheorem{thm}{Theorem}[section]

\newtheorem{lemma}[thm]{Lemma}
\newtheorem{prop}[thm]{Proposition}

\theoremstyle{definition}
\newtheorem{de}[thm]{Definition}
\theoremstyle{remark}
\newtheorem{rem}[thm]{Remark}

\newtheorem{ass}[thm]{Assumption}
\allowdisplaybreaks
\numberwithin{equation}{section}

\begin{document}

\title[The multiplicative ergodic theorem for McKean-Vlasov SDEs]
{The multiplicative ergodic theorem for McKean-Vlasov SDEs}

\author{Xianjin Cheng}
\address{X. Cheng: School of Mathematical Sciences,
Dalian University of Technology, Dalian 116024, P. R. China}
\email{xjcheng1119@hotmail.com}

\author{Zhenxin Liu}
\address{Z. Liu: School of Mathematical Sciences,
Dalian University of Technology, Dalian 116024, P. R. China}
\email{zxliu@dlut.edu.cn}

\author{Lixin Zhang}
\address{L. Zhang (Corresponding author): School of Mathematical Sciences,
Dalian University of Technology, Dalian 116024, P. R. China}
\email{lixinzhang8@hotmail.com;1031476306@mail.dlut.edu.cn}

\date{January 17, 2024}
\subjclass[2010]{34D08, 37A50, 37H15, 60H10.}
\keywords{McKean-Vlasov stochastic differential equations; Lyapunov exponents; the multiplicative ergodic theorem; the Lions derivative.}

\begin{abstract}
In this paper, we establish the multiplicative ergodic theorem for McKean-Vlasov stochastic differential equations, in which the Lyapunov exponent is defined using the upper limit. The reasonability of this definition is illustrated through an example; i.e., even when the coefficients are regular enough and their first-order derivatives are bounded, the upper limit cannot be replaced by a limit, as the limit may not exist. Furthermore, the example reveals how the dependence on distribution significantly influences the dynamics of the system and evidently distinguishes McKean-Vlasov stochastic differential equations from classical stochastic differential equations.
\end{abstract}

\maketitle

\section{Introduction}
McKean-Vlasov stochastic differential equations (SDEs), also known as the mean-field SDEs, were initially introduced by Kac \cite{Kac1,Kac2} in order to study the Boltzmann equation for the particle density in diluted monatomic gases and the stochastic toy model for the Vlasov kinetic equation of plasma. Such kinds of equations describe the motion of individual particles in a system, where each particle is influenced not only by its own position but also by all particles in the system. This leads to its solution not being a strong Markov process; see \cite{W}. McKean-Vlasov SDEs have been applied extensively in the fields of mean-field games, stochastic control and mathematical finance; see e.g. \cite{Carmona1,LL1,LL2,LL} and the references therein. The study of the well-posedness and the invariant measures for McKean-Vlasov SDEs has attracted extensive attention, and there have been a great deal of works, such as \cite{F,LM,Ma,MV,W}.

Lyapunov exponents play a crucial role in the study of the asymptotic behavior of dynamical systems, which are used to measure the separation rate between neighboring trajectories and describe the complexity of systems. It was introduced by the Russian mathematician Lyapunov in 1892 and has since gained widespread attention. In the realm of differential dynamical systems, Oseledets \cite{O} established the multiplicative ergodic theorem, which asserts the regularity and existence of the Lyapunov spectrum of a linear cocycle over a metric dynamical system under the assumption of an integrability condition. Since then, the theorem has attracted many researchers to provide new proofs and formulations with increasing generality; see e.g.  \cite{Ragh,LPD,Ruelle1,Ruelle2,Walters}. For classical SDEs, there has been a substantial amount of research on the Lyapunov exponent. For instance, Carverhill \cite{Car} applied the framework of ergodic theory and established the stochastic flow version of the multiplicative ergodic theorem under appropriate regularity conditions. Additional studies can be found in \cite{Cong, Liao1, Liao2}, and so on.

In this paper, we consider the $d$-dimensional McKean-Vlasov SDE with the deterministic initial value $x\in\R^d$:
\begin{align}\label{11}
X_t=x+\sum_{k=0}^{d'}\int_s^t V_k(X_r,P_{X_r})\,\rmd W_r^k,\quad t\geq s\geq0,
\end{align}
driven by a $d'$-dimensional standard Brownian motion $W_t=(W_t^1,...,W^{d'}_t)$, with coefficients $V_k:\R^d\times\cal{P}_2(\R^d)\rightarrow\R^d~(k=0,...,d')$, where $\rmd W_r^0=\rmd r$, $P_{X_r}$ is the law of $X_r$ and $\cal{P}_2(\R^d)$ is the set of all probability measures on $\R^d$ with finite second moments. The unique solution of equation \eqref{11} is denoted by $\Phi_{s,t}(x)$. When the coefficients satisfy certain regularity conditions, the mapping $x\mapsto\Phi_{s,t}(x)$ is $P$-$a.s.$ differentiable. The Jacobian matrix is denoted by $\partial_x\Phi_{s,t}(x)$. Notice that $\Phi_{s,t}(x)$ does not define a flow, as $\Phi_{r,t}(\Phi_{s,r}(x))$ is not the solution of \eqref{11} with $\Phi_{s,r}(x)$ as the initial value; see e.g.\ \cite{BLPR} for details. This results in our inability to discuss Lyapunov exponents of McKean-Vlasov SDEs in the same way as classical SDEs. In this paper, we define Lyapunov exponents of equation \eqref{11} with $s=0$ using the upper limit by
\begin{align}\label{upperLE}
\limsup_{t\rightarrow+\infty}\frac{1}{t}\log|\partial_x\Phi_{0,t}(\omega,x)v|=:\lambda(\omega,x,v),
\end{align}
for any $x\in\R^d$, $v\in\R^d$ and almost all $\omega$. For classical SDEs, under appropriate regularity assumptions, the upper limit in \eqref{upperLE} can be replaced by a limit, as the limit does exist. However, for McKean-Vlasov SDEs, we illustrate through an example in Section~\ref{section4} that even when the coefficients are regular enough and the first-order derivatives are bounded, the upper limit cannot be replaced by a limit, as the limit may not exist. In this example, the dependency of coefficients on the distribution is entirely determined by an ordinary differential equation (ODE), which transforms the autonomous McKean-Vlasov SDE into a non-autonomous classical SDE. By constructing coefficients of this ODE, we achieve alternating occurrences of rapid and slow growth for the particle distribution on a time partition with exponentially increasing interval lengths. Similar methods can be found in \cite{CLZ} or \cite{II}. This example demonstrates that the distribution has a significant impact on the dynamical behavior of the system. Meanwhile, this also illustrates the reasonability for employing the upper limit to define Lyapunov exponents in the mean-field version of the multiplicative ergodic theorem (Theorem \ref{MET1}) we are about to establish.

Unlike the classical multiplicative ergodic theorem, the absence of flow properties in this system prevents the application of classical conclusions from ergodic theory, such as random dynamical systems and invariant measures. Therefore, in the mean-field version of the multiplicative ergodic theorem presented in this paper, we approach the problem by considering trajectories individually. Furthermore, the absence of flow properties presents us with many challenges in the proof of the upper bound for all Lyapunov exponents. To overcome them, we adopt the approach presented in \cite{BLPR} to decompose equation~\eqref{11} with $s=0$ such that the solution that depends solely on the deterministic initial values is transformed into a form that simultaneously depends on both the deterministic initial values and their distributions, i.e. $\Phi_{0,t}(x)=\Phi_{0,t}^{\delta_x}(x)$. If we replace the initial distribution $\delta_x$ with $P_\xi$, then $\Phi_{0,t}^{P_\xi}(x)$ is the unique solution of decoupling SDE with the initial distribution $P_{\xi}$ starting from $x$ at $s$ and satisfies the property of a two-parameter stochastic flow with respect to the spatial variable, while the distribution $P_\xi$ is at the base point and moves with time. Moreover, the Jacobian matrix $\partial_x\Phi_{0,t}(x)$ can be expressed as the sum of the derivative of $\Phi_{0,t}^{\delta_x}(x)$ with respect to the spatial variable $x$ and the derivative of $\Phi_{0,t}^{\delta_x}(x)$ with respect to $x$ within the distribution, which are respectively denoted as $\partial_{x}\Phi_{0,t}^{\delta_y}(x)|_{y=x}$ and $\partial^{x}\Phi_{0,t}^{\delta_x}(y)|_{y=x}$.
Hence,
\begin{align*}
\limsup_{t\rightarrow+\infty}\frac{1}{t}\log|\partial_x\Phi_{0,t}(\omega,x)v|&=\limsup_{t\rightarrow+\infty}\frac{1}{t}\log|\partial_{x}\Phi_{0,t}^{\delta_x}(\omega,x)v+\partial^{x}\Phi_{0,t}^{\delta_x}(\omega,x)v|\\
&\leq\max\Big\{\limsup_{t\rightarrow+\infty}\frac{1}{t}\log|\partial_{x}\Phi_{0,t}^{\delta_x}(\omega,x)v|,\limsup_{t\rightarrow+\infty}\frac{1}{t}\log|\partial^{x}\Phi_{0,t}^{\delta_x}(\omega,x)v|\Big\}\\
&=:\max\big\{\lambda_1(\omega,x,v),\lambda_2(\omega,x,v)\big\}.
\end{align*}
To prove that $\lambda(\omega,x,v)$ has an upper bound with respect to $(\omega,x,v)$, it is sufficient to demonstrate that both $\lambda_1(\omega,x,v)$ and $\lambda_2(\omega,x,v)$ have upper bounds.
For $\lambda_1(\omega,x,v)$, since its corresponding equation is a classical SDE, the upper bound of $\lambda_1(\omega,x,v)$ can be obtained using the method in \cite{Cong} by applying the strong law of large numbers. However, for $\lambda_2(\omega,x,v)$, the shift in distribution prevents the satisfaction of the conditions for the strong law of large numbers. Nevertheless, by utilizing the properties of the logarithmic function and imitating the proof of the strong law of large numbers, we establish the upper bound of $\lambda_2(\omega,x,v)$.

This paper is organized as follows. In section \ref{Preliminaries}, we review and introduce some preliminaries. Section~\ref{MET} is dedicated to establishing the mean-field version of the multiplicative ergodic theorem. In Section~\ref{section4}, we construct an example to explain the reasonability for employing the upper limit in defining Lyapunov exponents in the preceding section and to reveal how the dependence on distribution significantly influences the dynamics of the system.

\section{Preliminaries}\label{Preliminaries}
\setcounter{equation}{0}
Throughout the paper, we assume that $\R^d$ is the $d$-dimensional Euclidean space with the Borel $\sigma$-field $\cal{B}(\R^d)$, $\cal{P}_2(\R^d)$ is the set of all probability measures $\mu$ over $(\R^d,\cal{B}(\R^d))$
with finite second moments, i.e.\ $\int_{\R^d} |x|^2\mu(\rmd x)<+\infty$, where $|\cdot|$ represents the norm of the Euclidean space. Additionally, let $\|\cdot\|$ be the $2$-norm of $d\times d$ matrices.

\subsection{Wiener Measure and Wiener Shift}

In this subsection, we introduce the concepts of Wiener measure and Wiener Shift. See \cite[Chapter 2]{Duan} for more details. Let $W_t$ be a standard, $d'$-dimensional Brownian motion defined on some probability space $(\Omega,\mathcal{F},P)$. Let $C(\R^+,\R^{d'})$ be the space of continuous functions from $\R^+$ to $\R^{d'}$, equipped with the metric
$$\rho(\omega_1,\omega_2)=\sum_{n=1}^{\infty}\frac{1}{2^n}\frac{\max_{0\leq t\leq n}|\omega_1(t)-\omega_2(t)|}{1+\max_{0\leq t\leq n}|\omega_1(t)-\omega_2(t)|}$$
for any $\omega_1,\omega_2\in C(\R^+,\R^{d'})$. This metric induces the topology of uniform convergence on compact time intervals. We denote the Borel $\sigma$-field generated by the metric $\rho$ as $\cal{B}(C(\R^+,\R^{d'}))$.
Moreover, $W_t$ can be thought of as a random variable with values in $C(\R^+,\R^{d'})$, denoted by $W$. Then the random variable $W$ induces a probability measure (i.e. the law of $W$) on $(C(\R^+,\R^{d'}),\cal{B}(C(\R^+,\R^{d'})))$, and we denote it as $P_W$. This measure is called the \emph{Wiener measure}.

We consider another stochastic process $\overline{W}_t$ on $C(\R^+,\R^{d'})$ by
$$\overline{W}_t:C(\R^+,\R^{d'})\rightarrow\R^{d'},\quad\overline{W}_t(\omega)=\omega(t).$$
It should be noted that the two stochastic processes $W_t$ and $\overline{W}_t$ are equivalent; i.e. they have the same finite-dimensional distributions. Thus $\overline{W}_t$ is a standard, $d'$-dimensional Brownian motion defined on the probability space $(C(\R^+,\R^{d'}),\cal{B}(C(\R^+,\R^{d'})),P_W)$. We refer to $\overline{W}_t$ as the \emph{canonical version} of $W_t$, and $(C(\R^+,\R^{d'}),\cal{B}(C(\R^+,\R^{d'})),P_W)$ as the \emph{canonical probability space} for Brownian motion $W_t$, respectively. For simplicity, we still denote the canonical probability space and Brownian motion $\overline{W}_t$ by the original notation $(\Omega,\mathcal{F},P)$ and $W_t$, which will be employed later in the subsequent sections.

We define the \emph{Wiener shift} $\theta_t$ on the canonical probability space $(\Omega,\mathcal{F},P)$ by
\begin{align*}
\theta_t:~&\Omega\rightarrow\Omega,~(\theta_t\omega)(s)=\omega(t+s)-\omega(t),\quad s\geq0,
\end{align*}
for any $t\geq0$. Consider the dynamical system $(\Omega,\mathcal{F},P,\theta_t)$. Then the Wiener measure $P$ is ergodic under the Wiener shift $\theta_t$; see \cite[Appendix A]{Arnold} for more details.

\subsection{Lions Derivatives and Regularity Conditions}

In this subsection, we introduce the concept of the Lions derivative and a class of functions defined on $\R^d\times\cal{P}_2(\R^d)$. See
\cite{BLPR}, \cite[Section~6]{CP} or \cite{CD1} for more details. Consider the ``rich enough" probability space $(\Omega, \cal{F}, P)$, which means that for any $\mu\in\cal{P}_2(\R^d)$ there exists a square-integrable random variable~$v$, defined on $(\Omega, \cal{F}, P)$ with values in $\R^d$, such that $P_v=\mu$, i.e.\ $\cal{P}_2(\R^d)=\{P_v;v\in L^2(\Omega,\cal{F},P;\R^d)\}$.
\begin{de}\label{LD}
We say that function $f:\cal{P}_2(\R^d)\rightarrow\R$ is \emph{Lions differentiable} at $P_{v}\in \cal{P}_2(\R^d)$ if $\widetilde{f}$ is Fr\'echet differentiable at $v$, where $\widetilde{f}$ is the canonical lift of $f$ from $\cal{P}_2(\R^d)$ to $L^2(\Omega,\cal{F},P;\R^d)$ satisfying $f(P_w)=\widetilde{f}(w)$ for all $w\in L^2(\Omega,\cal{F},P;\R^d)$.
\end{de}
It is worth noting that the above definition is independent of the selection of the probability space $(\Omega, \cal{F}, P)$ and the representation $v$, which implies the reasonability of Definition~\ref{LD}. See \cite{BLPR} or \cite{CD1} for more details. Then the \textit{Lions derivative} of $f$ at $P_{v}$ is denoted as $\partial_{\mu}f(P_{v})\in L^{2}_{P_{v}}(\R^d,\R^d):=\big\{g:\R^d\rightarrow\R^d;\int_{\R^d}|g(x)|^2P_v(\rmd x)<+\infty\big\}$, which is $P_v$-$a.s.$ uniquely determined. However, if we suppose that the Fr\'echet derivative $D\widetilde{f}:L^2(\Omega,\cal{F},P;\R^d)\rightarrow L(L^2(\Omega,\cal{F},P;\R^d);\R)$ is Lipschitz continuous, then there exists a Lipschitz continuous modification of $\partial_\mu f(P_v ,\cdot)$; see \cite{BLPR} or \cite[Lemma~3.3]{CD1} for details. Therefore, we can impose Lipschitz continuity or even differentiability requirements on $\partial_{\mu}f(P_{v},\cdot)$. In the following text, all notions related to Lions derivatives hold in the sense of `modification'.

For second-order Lions derivatives of $f:\cal{P}_2(\R^d)\rightarrow\R$, we can directly apply the previous discussion to each component of $\partial_{\mu}f(\cdot,y)=\big(\partial^1_{\mu}f(\cdot,y),...,\partial^d_{\mu}f(\cdot,y)\big)$ for fixed $y\in\R^d$; see e.g.\ \cite{Crisan} for details. Similarly, we can define higher-order Lions derivatives. If the $n$-order Lions derivative of $f$ exists, then for any multi-index $\alpha=(i_1,...,i_n)\in\{1,...,d\}^n$, we denote the $\alpha$-component of the $n$-order Lions derivative by $\partial^\alpha_\mu f=\partial^{i_n}_\mu...\partial^{i_1}_\mu f:\cal{P}_2(\R^d)\times(\R^d)^n\rightarrow\R$.

Let us finish the discussion about the Lions derivative by introducing a class of functions defined on $\R^d\times\cal{P}_2(\R^d)$ that satisfy the following assumptions. See \cite{Crisan} for details.
\begin{ass}[Regularity conditions]\label{con}
We say that $V= (V^1,...,V^d)\in\cal{C}_{b,Lip}^{1,1}(\R^d\times\cal{P}_2(\R^d);\R^d)$ if $\partial_\mu V$ and $\partial_x V$ exist and satisfy: there exists a constant $K>0$ such that
\begin{itemize}
  \item (Boundedness) for all $(x,\mu,v)\in \R^d\times\cal{P}_2(\R^d)\times\R^d$,
   $$\|\partial_x V(x,\mu)\|+\|\partial_\mu V(x,\mu,v)\|\leq K,$$
  \item (Lipschitz continuity) for all $(x,\mu,v),(x',\mu',v')\in \R^d\times\cal{P}_2(\R^d)\times\R^d$,
    $$\|\partial_x V(x,\mu)-\partial_x V(x',\mu')\|\leq K(|x-x'|+W_2(\mu,\mu')),$$
    $$\|\partial_\mu V(x,\mu,v)-\partial_\mu V(x',\mu',v')\|\leq K(|x-x'|+W_2(\mu,\mu')+|v-v'|),$$
    where $W_2$ is the 2-Wasserstein metric.
\end{itemize}
We say that $V\in\cal{C}_{b,Lip}^{n,n}(\R^d\times\cal{P}_2(\R^d);\R^d)$ if the following conditions are satisfied: for any $i\in\{1,...,d\}$, all multi-indices $\alpha\in\{1,...,d\}^k~(k=0,...,n),\gamma,\beta_1,...,\beta_{\#\alpha}\in\{(i_1,...,i_d);i_j\in\{0,...,n\}\}$ satisfying $\#\alpha+\nmid\!\beta_1\!\nmid+...+\nmid\!\beta_{\#\alpha}\!\nmid+\nmid\!\gamma\!\nmid\leq n$, the derivatives
$$\partial_x^\gamma\partial_{\boldsymbol{v}}^{\boldsymbol{\beta}}\partial_\mu^\alpha V^i(x,\mu,\boldsymbol{v}),
~\partial_{\boldsymbol{v}}^{\boldsymbol{\beta}}\partial_\mu^\alpha\partial_x^\gamma V^i(x,\mu,\boldsymbol{v}),
~\partial_{\boldsymbol{v}}^{\boldsymbol{\beta}}\partial_x^\gamma\partial_\mu^\alpha V^i(x,\mu,\boldsymbol{v})$$
exist and each of these derivatives is bounded and Lipschitz, where $\#\alpha$ is the number of components of $\alpha$, $\boldsymbol{\beta}=(\beta_1,...,\beta_{\#\alpha})$, $\boldsymbol{v}=(v_1,...,v_{\#\alpha})\in(\R^d)^{\#\alpha}$ and $\nmid\!\cdot\!\nmid$ represents the sum of its components. When $\#\alpha=0$, we define $\boldsymbol{\beta}=0$. Notice that $\partial_{\boldsymbol{v}}^{\boldsymbol{\beta}}$ and $\partial_\mu^\alpha$ are not exchangeable.
\end{ass}

\subsection{Cantelli' Strong Law of Large Numbers and the H\'ajek-R\'enyi Inequality}

In this subsection, we recall Cantelli' strong law of large numbers (see \cite[Theorem 3.1, Chapter 4]{Shi}) and the H\'ajek-R\'enyi inequality (see \cite{HR}), which play important roles in establishing that all Lyapunov exponents have an upper bound.

\begin{prop}[Cantelli' strong law of large numbers]
Let $\{\xi_n\}_{n\geq1}$ be a sequence of independent random variables with finite fourth moments and satisfy
$$E|\xi_n-E\xi_n|^4\leq C,~n\geq1,$$
for some constant $C$. Then as $n\rightarrow+\infty$,
$$\frac{S_n- ES_n}{n}\rightarrow0, ~P\text{-}a.s.,$$
where $S_n=\xi_1+...+\xi_n$.
\end{prop}

\begin{prop}[The H\'ajek-R\'enyi inequality]
Let $\{\xi_n\}_{n\geq1}$ be a sequence of independent random variables and $\{b_{n}, n \ge 1\}$ be a positive nondecreasing sequence of real numbers. Then for any $\varepsilon >0$ and positive integers $m<n$,
\begin{align*}
P \Biggl(\max_{m\leq k\leq n} \Biggl\vert \frac{1}{b_{k}}\sum _{i=1}^{k} (\xi_{i}-E\xi_i) \Biggr\vert > \varepsilon \Biggr)
\leq \varepsilon ^{-2} \Biggl(\frac{1}{b_{m}^{2}}\sum _{i=1}^{m}(\xi_{i}-E\xi_i)^{2}+\sum _{i=m+1}^{n} \frac{E(\xi_{i}-E\xi_i)^{2}}{b_{i}^{2}} \Biggr).
 \end{align*}
\end{prop}

\section{The Mean-Field Version of the Multiplicative Ergodic Theorem}\label{MET}

In this section, we will establish  the mean-field version of the multiplicative ergodic theorem. Consider the following McKean-Vlasov SDE starting from deterministic initial value $x\in\R^d$ at $s\geq0$:
\begin{align}\label{mvsdex1}
X_{t}=x+\sum_{k=0}^{d'}\int_s^t V_k(X_r,P_{X_r})\,\rmd W_r^k,\quad t\geq s,
\end{align}
where $P_{X_r}$ is the law of $X_r$, $W_t=(W_t^1,...,W^{d'}_t)$ is a $d'$-dimensional Brownian motion. Under the condition of $V_k\in \cal{C}_{b,Lip}^{1,1}(\R^d\times\cal{P}_2(\R^d);\R^d)(k=0,...,d')$, the existence and uniqueness of solutions to~\eqref{mvsdex1} is guaranteed, which is denoted by $\Phi_{s,t}(x)$. As mentioned earlier, $\Phi_{s,t}(x)$ does not define a flow. The absence of flow properties presents significant challenges. Therefore, before stating our main conclusions, we make some preparations.

\subsection{Some preparations}
We adopt the approach presented in \cite{BLPR} to decompose \eqref{mvsdex1} into two equations, namely, a McKean-Vlasov SDE with random initial value $\xi\in L^2(\Omega,\cal{F},P;\R^d)$ independent of $W_t~(t\geq s)$:
\begin{align}\label{mvsde11}
\Phi_{s,t}(\xi)=\xi+\sum_{k=0}^{d'}\int_s^t V_k(\Phi_{s,r}(\xi),P_{\Phi_{s,r}(\xi)})\,\rmd W_r^k,
\end{align}
and a decoupling classical SDE with the initial distribution $P_{\xi}$ starting from $x$ at $s$:
\begin{align}\label{dsde11}
\Phi_{s,t}^{P_\xi}(x)=x+\sum_{k=0}^{d'}\int_s^t V_k(\Phi_{s,r}^{P_\xi}(x),P_{\Phi_{s,r}(\xi)})\,\rmd W_r^k,
\end{align}
where $t\geq s\geq0$. The condition of $V_k\in \cal{C}_{b,Lip}^{1,1}(\R^d\times\cal{P}_2(\R^d);\R^d)(k=0,...,d')$ implies the existence and uniqueness of solutions to \eqref{mvsde11} and \eqref{dsde11}.
The uniqueness of the solution of the both equations and \cite[Lemma 3.1]{BLPR} imply the following flow property:
\begin{align}\label{flow1}
\Big(\Phi_{r,t}^{P_{r-s}^*P_{\xi}}(\Phi_{s,r}^{P_\xi}(x)),P_{t-r}^*P_{r-s}^*P_{\xi}\Big)=\Big(\Phi_{s,t}^{P_\xi}(x),P_{t-s}^*P_{\xi}\Big),\quad  s\leq r\leq t,
\end{align}
where $P^*_{t}$ is the semigroup of transfer operators generated by autonomous system \eqref{mvsde11}. For McKean-Vlasov SDEs, the solutions and their distributions exist in pairs and are of equal importance. However, we are more concerned with spatial variables in this paper. It can be anticipated that the distribution has a significant impact on our work. Specifically, according to \eqref{flow1}, we consider the following property when $P_\xi=\delta_y$ and $s=0$:
\begin{align}\label{flow}
\Phi_{r,t}^{P_{r}^*\delta_y}(\Phi_{0,r}^{\delta_y}(x))=\Phi_{0,t}^{\delta_y}(x),\quad 0\leq r\leq t.
\end{align}
Moreover, it follows from the uniqueness of the solution of equation \eqref{dsde11} that $\Phi_{0,t}^{P_\xi}(x)|_{x=\xi}=\Phi_{0,t}(\xi)$, which implies that the solution $\Phi_{0,t}(x)$ to \eqref{mvsdex1} with $s=0$ is equal to $\Phi_{0,t}^{\delta_x}(x)$.

Furthermore, let us discuss the differentiability of the mapping $x\mapsto\Phi_{0,t}^{\delta_x}(x)$. In fact, we need to consider a more general case: $x\mapsto\Phi_{s,t}^{P^*_{s}\delta_x}(x)$ $(0\leq s\leq t)$. This requires us to separately consider the differentiability of $x\mapsto\Phi_{s,t}^{P^*_{s}\delta_y}(x)$ and $x\mapsto\Phi_{s,t}^{P^*_{s}\delta_x}(y)$ for any $y\in\R^d$. It should be noted that $\Phi_{s,t}^{P^*_{s}\delta_y}(x)$ is the unique solution of the following decoupling SDE with the initial distribution $P^*_{s}\delta_y$ starting from $x$ at $s\geq0$:
\begin{align}\label{dsde111}
\Phi_{s,t}^{P^*_{s}\delta_y}(x)=x+\sum_{k=0}^{d'}\int_s^t V_k(\Phi_{s,r}^{P^*_{s}\delta_y}(x),P_{\Phi_{0,r}(y)})\,\rmd W_r^k,\quad t\geq s.
\end{align}
By interchanging the positions of $x$ and $y$ in equation \eqref{dsde111}, we obtain the equation satisfied by $\Phi_{s,t}^{P^*_{s}\delta_x}(y)$. Under the condition of $V_k\in \cal{C}_{b,Lip}^{1,1}(\R^d\times\cal{P}_2(\R^d);\R^d)(k=0,...,d')$, the mapping $x\mapsto\Phi_{s,t}^{P^*_{s}\delta_y}(x)$ is $P$-$a.s.$ differentiable for any $y\in\R^d$ and its Jacobian matrix is denoted as $\partial_{x}\Phi_{s,t}^{P^*_{s}\delta_y}(x)$ satisfying the SDE
\begin{align}\label{DE1}
\partial_{x}\Phi_{s,t}^{P_{s}^*\delta_y}(x)&=I+\sum_{k=0}^{d'}\int_{s}^t
\partial_xV_k(\Phi_{s,r}^{P_{s}^*\delta_y}(x),P_{\Phi_{0,r}(y)})\,
\partial_{x}\Phi_{s,r}^{P_{s}^*\delta_y}(x)\,\rmd W^k_r,
\end{align}
where $I:=I_{d\times d}$ is the identity matrix; see \cite[Theorem 3.1, Chapter II]{Kunita} for details. Now, we are going to prove the differentiability of the mapping $x\mapsto\Phi_{s,t}^{P^*_{s}\delta_x}(y)$ for any $s<t$, $y\in\R^d$.
\begin{lemma}\label{lemma3.4}
Assume that $V_0,...,V_{d'}$ belong to $\cal{C}_{b,Lip}^{1,1}(\R^d\times\cal{P}_2(\R^d);\R^d)$. Then there exists a modification of $\Phi_{s,t}^{P^*_{s}\delta_x}(y)$, with the same notation, such that for any $s<t,y\in\R^d$, the mapping $x\mapsto\Phi_{s,t}^{P^*_{s}\delta_x}(y)$ is $P$-$a.s.$ differentiable, and its Jacobian matrix, denoted as $\partial^{x}\Phi_{s,t}^{P^*_{s}\delta_x}(y)$, satisfies the following SDE
\begin{align*}
\partial^{x}\Phi_{s,t}^{P_{s}^*\delta_x}(y)&=\sum_{k=0}^{d'}\int_{s}^t\partial_xV_k(\Phi_{s,r}^{P_{s}^*\delta_x}(y),P_{\Phi_{0,r}(x)})\,
\partial^{x}\Phi_{s,r}^{P_{s}^*\delta_x}(y)\,\rmd W^k_r\\
&\quad+\sum_{k=0}^{d'}\int_{s}^t\widetilde{E}\Big[\partial_\mu V_k(\Phi_{s,r}^{P_{s}^*\delta_x}(y),P_{\Phi_{0,r}(x)},\widetilde{\Phi}_{0,r}(x))\,
\partial_x\widetilde{\Phi}_{0,r}(x)\Big]\,\rmd W^k_r.
\end{align*}
\end{lemma}
\begin{proof}
Since the proof is similar to that of Theorem 3.3 in \cite{CL}, we omit it.
\end{proof}

\begin{rem}
\begin{enumerate}
  \item It should be noted that the derivative of $\Phi_{s,t}^{P_{\xi}}(x)$ with respect to the distribution is in the $L^2$ sense;
  see \cite[Lemma 4.2]{BLPR}. However, since we are considering deterministic initial values, according to Kolmogorov's continuity theorem, we can obtain the differentiability of $x\mapsto\Phi_{s,t}^{P^*_{s}\delta_x}(y)$ in the pathwise sense.
  \item Lemma \ref{lemma3.4} suggests that the Jacobian matrix $\partial_x\Phi_{0,t}(x)$ satisfies
\begin{align}\label{derivatives}
\partial_x\Phi_{0,t}(x)=\partial_x\big[\Phi_{0,t}^{\delta_x}(x)\big]=\partial_x\Phi_{0,t}^{\delta_y}(x)|_{y=x}+\partial^{x}\Phi_{0,t}^{\delta_x}(y)|_{y=x},~~P\text{-}a.s.
\end{align}
\end{enumerate}
\end{rem}

\begin{lemma}\label{lem3.6}
Under the same assumptions as in Lemma \ref{lemma3.4}, for any $p\geq2$, there exists a constant $\widehat{C}(p,K)>0$ such that for any $n\in\mathbb{N},x,y\in\R^d$,
\begin{align*}
&E\sup_{n-1\leq t\leq n}\bigr\|\partial^{x}\Phi_{n-1,t}^{P_{n-1}^*\delta_x}\bigr(\Phi_{0,n-1}^{\delta_x}(y)\bigr)\bigr\|^p\leq \widehat{C}(p,K)^n.
\end{align*}
Specifically, we denote $\widehat{C}(2, K)$ as $C^*$.
\end{lemma}
\begin{proof}
For the sake of simplicity, assume $d'=1$. We first estimate the $p$-moment of the Jacobian matrix $\partial_x\Phi_{0,t}(x)$. For any $i\in\{1,...,d\}$, $e_{i}$ is the unit vector $(0,...,0,1,0,...,0)$ with $1$ as the $i$-th component. Then $\partial_x\Phi_{0,t}(x)e_i$ satisfies
\begin{align*}
\partial_x\Phi_{0,t}(x)e_i=e_i&+\sum_{k=0}^1\int_0^t \partial_x V_k(\Phi_{0,r}(x),P_{\Phi_{0,r}(x)})\cdot \partial_x\Phi_{0,r}(x)e_i\,\rmd W_r^k\\
&+\sum_{k=0}^1\int_0^t\widetilde{E}\bigr[\partial_{\mu}V_k(\Phi_{0,r}(x),P_{\Phi_{0,r}(x)},\widetilde{\Phi}_{0,r}(x))\cdot
\partial_x\widetilde{\Phi}_{0,r}(x)e_i\bigr]\,\rmd W_r^k.
\end{align*}
See \cite{CL} for more details. Applying the It\^o formula, we have
\begin{align*}
|\partial_x\Phi&_{0,t}(x)e_i|^p\\
=1&+p\int_0^t|\partial_x\Phi_{0,r}(x)e_i|^{p-2}\Big\langle\partial_x\Phi_{0,r}(x)e_i,\partial_x V_0(\Phi_{0,r}(x),P_{\Phi_{0,r}(x)})\cdot \partial_x\Phi_{0,r}(x)e_i\\
&\qquad\qquad\qquad\qquad\qquad+\widetilde{E}\bigr[\partial_{\mu}V_0(\Phi_{0,r}(x),P_{\Phi_{0,r}(x)},\widetilde{\Phi}_{0,r}(x))\cdot
\partial_x\widetilde{\Phi}_{0,r}(x)e_i\bigr]\Big\rangle \rmd r\\
&+p\int_0^t|\partial_x\Phi_{0,r}(x)e_i|^{p-2}\Big\langle\partial_x\Phi_{0,r}(x)e_i,\partial_x V_1(\Phi_{0,r}(x),P_{\Phi_{0,r}(x)})\cdot \partial_x\Phi_{0,r}(x)e_i\\
&\qquad\qquad\qquad\qquad\qquad+\widetilde{E}\bigr[\partial_{\mu}V_1(\Phi_{0,r}(x),P_{\Phi_{0,r}(x)},\widetilde{\Phi}_{0,r}(x))\cdot
\partial_x\widetilde{\Phi}_{0,r}(x)e_i\bigr]\,\rmd W_r^1\Big\rangle\\
&+\frac{p}{2}\int_0^t|\partial_x\Phi_{0,r}(x)e_i|^{p-2}\Big|\partial_x V_1(\Phi_{0,r}(x),P_{\Phi_{0,r}(x)})\cdot \partial_x\Phi_{0,r}(x)e_i\\
&\qquad\qquad\qquad\qquad\qquad+\widetilde{E}\bigr[\partial_{\mu}V_1(\Phi_{0,r}(x),P_{\Phi_{0,r}(x)},\widetilde{\Phi}_{0,r}(x))\cdot
\partial_x\widetilde{\Phi}_{0,r}(x)e_i\bigr]\Big|^2\rmd r\\
&+\frac{p(p-2)}{2}\int_0^t|\partial_x\Phi_{0,r}(x)e_i|^{p-4}\Big|\Big(\partial_x\Phi_{0,r}(x)e_i\Big)^T\Big(\partial_x V_1(\Phi_{0,r}(x),P_{\Phi_{0,r}(x)})\cdot \partial_x\Phi_{0,r}(x)e_i\\
&\qquad\qquad\qquad\qquad\qquad+\widetilde{E}\bigr[\partial_{\mu}V_1(\Phi_{0,r}(x),P_{\Phi_{0,r}(x)},\widetilde{\Phi}_{0,r}(x))\cdot
\partial_x\widetilde{\Phi}_{0,r}(x)e_i\bigr]\Big)\Big|^2\rmd r.
\end{align*}
After a simple calculation, we obtain
\begin{align*}
E&\sup_{t\in[0,n]}|\partial_x\Phi_{0,t}(x)e_i|^p\\
=&1+pE\sup_{t\in[0,n]}\int_0^t|\partial_x\Phi_{0,r}(x)e_i|^{p-2}\Big\langle\partial_x\Phi_{0,r}(x)e_i,\partial_x V_0(\Phi_{0,r}(x),P_{\Phi_{0,r}(x)})\cdot \partial_x\Phi_{0,r}(x)e_i\\
&\qquad\qquad\qquad\qquad\qquad+\widetilde{E}\bigr[\partial_{\mu}V_0(\Phi_{0,r}(x),P_{\Phi_{0,r}(x)},\widetilde{\Phi}_{0,r}(x))\cdot
\partial_x\widetilde{\Phi}_{0,r}(x)e_i\bigr]\Big\rangle \rmd r\\
&+pE\sup_{t\in[0,n]}\int_0^t|\partial_x\Phi_{0,r}(x)e_i|^{p-2}\Big\langle\partial_x\Phi_{0,r}(x)e_i,\partial_x V_1(\Phi_{0,r}(x),P_{\Phi_{0,r}(x)})\cdot \partial_x\Phi_{0,r}(x)e_i\\
&\qquad\qquad\qquad\qquad\qquad+\widetilde{E}\bigr[\partial_{\mu}V_1(\Phi_{0,r}(x),P_{\Phi_{0,r}(x)},\widetilde{\Phi}_{0,r}(x))\cdot
\partial_x\widetilde{\Phi}_{0,r}(x)e_i\bigr]\,\rmd W_r^1\Big\rangle\\
&+\frac{p}{2}E\sup_{t\in[0,n]}\int_{0}^t|\partial_x\Phi_{0,r}(x)e_i|^{p-2}\Big|\partial_x V_1(\Phi_{0,r}(x),P_{\Phi_{0,r}(x)})\cdot \partial_x\Phi_{0,r}(x)e_i\\
&\qquad\qquad\qquad\qquad\qquad+\widetilde{E}\bigr[\partial_{\mu}V_1(\Phi_{0,r}(x),P_{\Phi_{0,r}(x)},\widetilde{\Phi}_{0,r}(x))\cdot
\partial_x\widetilde{\Phi}_{0,r}(x)e_i\bigr]\Big|^2\rmd r\\
&+\frac{p(p-2)}{2}E\sup_{t\in[0,n]}\int_{0}^t|\partial_x\Phi_{0,r}(x)e_i|^{p-4}\Big|\Big(\partial_x\Phi_{0,r}(x)e_i\Big)^T\Big(\partial_x V_1(\Phi_{0,r}(x),P_{\Phi_{0,r}(x)})\cdot \partial_x\Phi_{0,r}(x)e_i\\
&\qquad\qquad\qquad\qquad\qquad+\widetilde{E}\bigr[\partial_{\mu}V_1(\Phi_{0,r}(x),P_{\Phi_{0,r}(x)},\widetilde{\Phi}_{0,r}(x))\cdot
\partial_x\widetilde{\Phi}_{0,r}(x)e_i\bigr]\Big)\Big|^2\rmd r\\
\leq&1+2\big(p(p-1)K^2+pK\big)\int_0^nE\sup_{u\in[0,r]}|\partial_x\Phi_{0,u}(x)e_i|^{p} \rmd r\\
&+pE\sup_{t\in[0,n]}\Big|\int_0^t|\partial_x\Phi_{0,r}(x)e_i|^{p-2}\Big\langle\partial_x\Phi_{0,r}(x)e_i,\partial_x V_1(\Phi_{0,r}(x),P_{\Phi_{0,r}(x)})\cdot \partial_x\Phi_{0,r}(x)e_i\\
&\qquad\qquad\qquad\qquad\qquad+\widetilde{E}\bigr[\partial_{\mu}V_1(\Phi_{0,r}(x),P_{\Phi_{0,r}(x)},\widetilde{\Phi}_{0,r}(x))\cdot
\partial_x\widetilde{\Phi}_{0,r}(x)e_i\bigr]\,\rmd W_r^1\Big\rangle\Big|.
\end{align*}
We now estimate the last term, obtaining
\begin{align*}
&pE\sup_{t\in[0,n]}\Bigr|\int_{0}^t|\partial_x\Phi_{0,r}(x)e_i|^{p-2}\Big\langle\partial_x\Phi_{0,r}(x)e_i,\partial_x V_1(\Phi_{0,r}(x),P_{\Phi_{0,r}(x)})\cdot \partial_x\Phi_{0,r}(x)e_i\\
&\qquad\qquad\qquad\qquad\qquad+\widetilde{E}\bigr[\partial_{\mu}V_1(\Phi_{0,r}(x),P_{\Phi_{0,r}(x)},\widetilde{\Phi}_{0,r}(x))\cdot
\partial_x\widetilde{\Phi}_{0,r}(x)e_i\bigr]\,\rmd W_r^1\Big\rangle\Bigr|\\
\leq&pE\Bigr(\Bigr|\int_{0}^n|\partial_x\Phi_{0,r}(x)e_i|^{p-2}\Big\langle\partial_x\Phi_{0,r}(x)e_i,\partial_x V_1(\Phi_{0,r}(x),P_{\Phi_{0,r}(x)})\cdot \partial_x\Phi_{0,r}(x)e_i\\
&\qquad\qquad\qquad\qquad\qquad+\widetilde{E}\bigr[\partial_{\mu}V_1(\Phi_{0,r}(x),P_{\Phi_{0,r}(x)},\widetilde{\Phi}_{0,r}(x))\cdot
\partial_x\widetilde{\Phi}_{0,r}(x)e_i\bigr]\,\rmd W_r^1\Big\rangle\Bigr|^2\Bigr)^{\frac{1}{2}}\\
=&pE\Bigr(\int_{0}^n|\partial_x\Phi_{0,r}(x)e_i|^{2p-4}\Bigr|\Big\langle\partial_x\Phi_{0,r}(x)e_i,\partial_x V_1(\Phi_{0,r}(x),P_{\Phi_{0,r}(x)})\cdot \partial_x\Phi_{0,r}(x)e_i\\
&\qquad\qquad\qquad\qquad\qquad+\widetilde{E}\bigr[\partial_{\mu}V_1(\Phi_{0,r}(x),P_{\Phi_{0,r}(x)},\widetilde{\Phi}_{0,r}(x))\cdot
\partial_x\widetilde{\Phi}_{0,r}(x)e_i\bigr]\Big\rangle\Bigr|^2\,\rmd r\Bigr)^{\frac{1}{2}}\\
\leq&pE\Bigr[\sup_{t\in[0,n]}|\partial_x\Phi_{0,r}(x)e_i|^\frac{p}{2}\Bigr(\int_{0}^n|\partial_x\Phi_{0,r}(x)e_i|^{p-4}\Bigr|\Big\langle\partial_x\Phi_{0,r}(x)e_i,\partial_x V_1(\Phi_{0,r}(x),P_{\Phi_{0,r}(x)})\cdot \partial_x\Phi_{0,r}(x)e_i\\
&\qquad\qquad\qquad\qquad\qquad+\widetilde{E}\bigr[\partial_{\mu}V_1(\Phi_{0,r}(x),P_{\Phi_{0,r}(x)},\widetilde{\Phi}_{0,r}(x))\cdot
\partial_x\widetilde{\Phi}_{0,r}(x)e_i\bigr]\Big\rangle\Bigr|^2\,\rmd r\Bigr)^{\frac{1}{2}}\Bigr]\\
\leq&\frac{E\sup_{t\in[0,n]}|\partial_x\Phi_{0,r}(x)e_i|^p}{2}+\frac{p^2}{2}E\Bigr(\int_{0}^n|\partial_x\Phi_{0,r}(x)e_i|^{p-4}\Bigr|\Big\langle\partial_x\Phi_{0,r}(x)e_i,\\
&\quad\partial_x V_1(\Phi_{0,r}(x),P_{\Phi_{0,r}(x)})\cdot \partial_x\Phi_{0,r}(x)e_i+\widetilde{E}\bigr[\partial_{\mu}V_1(\Phi_{0,r}(x),P_{\Phi_{0,r}(x)},\widetilde{\Phi}_{0,r}(x))\cdot
\partial_x\widetilde{\Phi}_{0,r}(x)e_i\bigr]\Big\rangle\Bigr|^2\,\rmd r\Bigr)\\
\leq&\frac{E\sup_{t\in[0,n]}|\partial_x\Phi_{0,r}(x)e_i|^p}{2}+2p^2K^2\int_{0}^nE\sup_{u\in[0,r]}|\partial_x\Phi_{0,u}(x)e_i|^{p}\rmd r.
\end{align*}
Further,
\begin{align*}
E\sup_{t\in[0,n]}|\partial_x\Phi_{0,t}(x)e_i|^p\leq1+4\Big(p(p-1)K^2+2p^2K^2+pK\Big)\int_0^nE\sup_{u\in[0,r]}|\partial_x\Phi_{0,u}(x)e_i|^{p} \rmd r.
\end{align*}
Set
\begin{align}\label{constant}
4\Big(p(p-1)K^2+2p^2K^2+pK\Big)=:C(p,K).
\end{align}
The Gronwall inequality implies
\begin{align*}
&E\sup_{t\in[0,n]}|\partial_x\Phi_{0,t}(x)e_i|^p
\leq e^{C(p,K)n}.
\end{align*}
Therefore, for any $n\geq1$,
\begin{align}\label{p-estimate}
E\sup_{t\in[n-1,n]}\|\partial_x\Phi_{0,t}(x)\|^p\leq C(d,p)e^{C(p,K)n},
\end{align}
where $C(d,p)>0$.

Next, we proceed to estimate the $p$-moment of $\partial^{x}\Phi_{n-1,t}^{P_{n-1}^*\delta_x}(z)$. It follows from Lemma \ref{lemma3.4} that
\begin{align*}
\partial^{x}\Phi_{n-1,t}^{P_{n-1}^*\delta_x}(z)&=\sum_{k=0}^1\int_{n-1}^t\partial_{x}V_k(\Phi_{n-1,r}^{P_{n-1}^*\delta_x}(z),P_{\Phi_{0,r}(x)})\,\partial^{x}\Phi_{n-1,r}^{P_{n-1}^*\delta_x}(z)\,\rmd W_r^k\\
&\quad+\sum_{k=0}^1\int_{n-1}^t\widetilde{E}\Big[\partial_\mu V_k(\Phi_{n-1,r}^{P_{n-1}^*\delta_x}(z),P_{\Phi_{0,r}(x)},\widetilde{\Phi}_{0,r}(x))\,\partial_x\widetilde{\Phi}_{0,r}(x)\Big]\,\rmd W_r^k.
\end{align*}
By the martingale inequality, the BDG inequality, the H\"older inequality and the boundedness of the coefficients, we obtain that there exists a constant $C'(p,K,d)>0$ such that
\begin{align*}
&E\sup_{n-1\leq t\leq n}\bigr\|\partial^{x}\Phi_{n-1,t}^{P_{n-1}^*\delta_x}(z)\bigr\|^p\\
&\leq C'(p,K,d)e^{C(p,K)n}+C'(p,K,d)\int_{n-1}^nE\sup_{n-1\leq u\leq r}\bigr\|\partial^{x}\Phi_{n-1,u}^{P_{n-1}^*\delta_x}(z)\bigr\|^p\rmd r.
\end{align*}
It follows from the Gronwall inequality that
\begin{align*}
E\sup_{n-1\leq t\leq n}\bigr\|\partial^{x}\Phi_{n-1,t}^{P_{n-1}^*\delta_x}(z)\bigr\|^p\leq C'(p,K,d)e^{C(p,K)n}e^{\int_{n-1}^nC'(p,K,d)\rmd r}\leq \widetilde{C}(p,K,d)e^{C(p,K)n}.
\end{align*}

Hence,
\begin{align*}
&E\sup_{n-1\leq t\leq n}\bigr\|\partial^{x}\Phi_{n-1,t}^{P_{n-1}^*\delta_x}(\Phi_{0,n-1}^{\delta_x}(y))\bigr\|^p\\
&\leq E\Big[E\sup_{n-1\leq t\leq n}\bigr\|\partial^{x}\Phi_{n-1,t}^{P_{n-1}^*\delta_x}(z)\bigr\|^p_{z=\Phi_{0,n-1}^{\delta_x}(y)}\Big]\leq \widetilde{C}(p,K,d)e^{C(p,K)n}=:\widehat{C}(p,K,d)^n.
\end{align*}
The proof is complete.
\end{proof}

Below, we state another lemma used in the proof of our main conclusion. Consider the decoupling SDE with the initial distribution $\delta_y$ starting from $x$ at $0$:
\begin{align}\label{dsdex}
\Phi_{0,t}^{\delta_y}(x)=x+\sum_{k=0}^{d'}\int_0^tV_k(\Phi_{0,r}^{\delta_y}(x),P_{\Phi_{0,r}(y)})\rmd W_r^k.
\end{align}
As mentioned earlier, under the condition of $V_k\in \cal{C}_{b,Lip}^{1,1}(\R^d\times\cal{P}_2(\R^d);\R^d)(k=0,...,d')$, the mapping $x\mapsto\Phi_{0,t}^{\delta_y}(x)$ is $P$-$a.s.$ differentiable, and its Jacobian matrix is denoted by $\partial_{x}\Phi_{0,t}^{\delta_y}(x)$.  For fixed $x,y\in\R^d$, we set
$$A_k(t):=\partial_{x}V_k(\Phi_{0,t}^{\delta_y}(x),P_{\Phi_{0,t}(y)}).$$
Consider the following linear SDE:
\begin{align*}
Y_t&=I+\sum_{k=0}^{d'}\int_s^tA_k(r)\,Y_r\,\rmd W_r^k,\quad 0\leq s\leq t.
\end{align*}
We denote the fundamental solution matrix by $\Psi_{s,t}(s\leq t)$. Then it satisfies that $\Psi_{s,t}=\partial_{x}\Phi_{s,t}^{P_s^*\delta_y}(\Phi_{0,s}^{\delta_y}(x))$ and $\Psi_{s,t}=\Psi_{r,t}\cdot\Psi_{s,r}$ with $0\leq s\leq r\leq t$.

\begin{lemma}\label{lem2}
For the two-parameter flow $\Psi_{s,t}$, there exists $M>0$ such that for any $n\in\mathbb{N}$, the following inequality holds:
$$P\Big(\Big\{\omega\in\Omega|\sup_{n\leq t\leq n+1}\|\Psi_{n,t}(\omega)\|\geq n+d^2e^K\Big\}\Big)\leq \frac{M}{n^2}.$$
\end{lemma}
\begin{proof}
The lemma follows directly from \cite[Lemma 3.3]{Cong}.
\end{proof}

\subsection{Main results}
After considering the aforementioned facts, let us begin establishing the multiplicative ergodic theorem. Consider the following McKean-Vlasov SDE with deterministic initial value $x\in\R^d$:
\begin{align}\label{mvsdex}
X_{t}=x+\sum_{k=0}^{d'}\int_0^t V_k(X_r,P_{X_r})\,\rmd W_r^k.
\end{align}
As mentioned earlier, under the condition of $V_k\in \cal{C}_{b,Lip}^{1,1}(\R^d\times\cal{P}_2(\R^d);\R^d)(k=0,...,d')$, the existence and uniqueness of solutions to \eqref{mvsdex} can be guaranteed, still denoted as $\Phi_{0,t}(x)$. Moreover, the mapping $x\mapsto\Phi_{0,t}(x)$ is $P$-$a.s.$ differentiable. Its Jacobian matrix is denoted as $\partial_x\Phi_{0,t}(x)$. Then the theorem is stated as follows:

\begin{thm}[The mean-field version of the multiplicative ergodic theorem]\label{MET1}
Assume that coefficients $V_0,...,V_{d'}$ in equation~\eqref{mvsdex} belong to $\cal{C}_{b,Lip}^{1,1}(\R^d\times\cal{P}_2(\R^d);\R^d)$. Then for some constant $\kappa>0$, the following facts hold: for any $x\in\R^d$, there exists a set $\Omega(x)\subset\Omega$ of full $P$-measure such that for any $\omega\in\Omega(x)$, there exist $r(\omega,x)~(\leq d)$ linear subspaces of $\R^d$
$$\{0\}=V_0\subset V_1(\omega,x)\subset...\subset V_{r(\omega,x)}(\omega,x)=\R^d$$
and numbers
$$-\infty\leq\lambda_1(\omega,x)<\lambda_2(\omega,x)<...<\lambda_{r(\omega,x)}(\omega,x)\leq\kappa$$
such that
for any $v\in V_i(\omega,x)\setminus V_{i-1}(\omega,x)~(1\leq i\leq r(\omega,x))$,
\begin{align}\label{LE}
\limsup_{t\rightarrow+\infty}\frac{1}{t}\log|\partial_x\Phi_{0,t}(\omega,x)v|=\lambda_i(\omega,x).
\end{align}
\end{thm}
\begin{rem}
\begin{enumerate}
  \item In Section \ref{section4}, we will construct an example to demonstrate that even when the coefficients are regular enough and the first-order derivatives are bounded, the upper limit in \eqref{LE} cannot be replaced by a limit, as the limit may not exist. Thus, this explains the reasonability for employing the upper limit to define Lyapunov exponents in this theorem.
  \item The absence of flow properties renders classical theoretical methods unusable. The invariance of the mapping $\omega\mapsto\lambda_i(\omega,x)$ with respect to Wiener shift $\theta_t$ cannot be obtained. Therefore, even though system $(\Omega,\mathcal{F},P,\theta_t)$ is ergodic, whether the Lyapunov exponent $\lambda_i(\omega,x)$ is independent of $\omega$ remains unknown.
  \item The absence of flow properties, unlike in the classical multiplicative ergodic theorem, results in the classical conclusions (such as random dynamical systems and invariant measures) of ergodic theory being inapplicable. Therefore, we approach the problem by considering trajectories individually in this theorem.
\end{enumerate}
\end{rem}

In fact, all we need to do is to establish the existence of the upper bound~$\kappa$ for all Lyapunov exponents and the existence of the full $P$-measurable set
$\Omega(x)$ for any $x\in\R^d$.
\begin{thm}\label{Boundedness}
Assume that coefficients $V_0,...,V_{d'}$ in equation~\eqref{mvsdex} belong to $\cal{C}_{b,Lip}^{1,1}(\R^d\times\cal{P}_2(\R^d);\R^d)$. Then there exists a constant $\kappa>0$ such that for any $x\in\R^d$, there exists a set $\Omega(x)\subset\Omega$ of full $P$-measure such that
\begin{align*}
\lambda(\omega,x,v):=\limsup_{t\rightarrow+\infty}\frac{1}{t}\log|\partial_x\Phi_{0,t}(\omega,x)v|\leq \kappa
\end{align*}
holds for any $\omega\in \Omega(x)$ and any tangent vector $v\in \R^d$.
\end{thm}

\begin{proof}
For given $x,y\in\R^d$, set
$$B_n(x,y):=\Big\{\omega\in\Omega|\sup_{n\leq t\leq n+1}\|\partial_{x}\Phi_{n,t}^{P_n^*\delta_y}(\omega,\Phi_{0,n}^{\delta_y}(\omega,x))\|\geq n+d^2e^K\Big\},$$
$$B(x,y):=\bigcup_{k=1}^{+\infty}\bigcap_{j=k}^{+\infty}(\Omega\setminus B_j(x,y)).$$
Then Lemma \ref{lem2} implies
$$P\Big(B_n(x,y)\Big)\leq\frac{M}{n^2}.$$
Therefore,  $$\sum_{n=1}^{+\infty}P\Big(B_n(x,y)\Big)<+\infty.$$
By the Borel-Cantelli Lemma, it holds that $P(B(x,y))=1$.
Hence, for any $\omega\in B(x,y)$, there exists $k\geq1$ such that for any $j\geq k$, $\omega$ belongs to $\Omega\setminus B_j(x,y)$, i.e.
\begin{align}\label{12}
\sup_{j\leq t\leq j+1}\|\partial_{x}\Phi_{j,t}^{P_j^*\delta_y}(\omega,\Phi_{0,j}^{\delta_y}(\omega,x))\|\leq j+d^2e^K.
\end{align}

For any $x\in \R^d$ and any tangent vector $v\in \R^d$, it follows from \eqref{derivatives} that
\begin{align*}
\frac{1}{t}\log|\partial_x\Phi_{0,t}(x)v|&=\frac{1}{t}\log|\partial_{x}\Phi_{0,t}^{\delta_x}(x)v+\partial^{x}\Phi_{0,t}^{\delta_x}(x)v|\\
&\leq\frac{1}{t}\log\Big(|\partial_{x}\Phi_{0,t}^{\delta_x}(x)v|+|\partial^{x}\Phi_{0,t}^{\delta_x}(x)v|\Big)\nonumber\\
&\leq\frac{1}{t}\log\Big(2\max\big\{|\partial_{x}\Phi_{0,t}^{\delta_x}(x)v|,|\partial^{x}\Phi_{0,t}^{\delta_x}(x)v|\big\}\Big)\nonumber\\
&=\max\Big\{\frac{1}{t}\log|\partial_{x}\Phi_{0,t}^{\delta_x}(x)v|,\frac{1}{t}\log|\partial^{x}\Phi_{0,t}^{\delta_x}(x)v|\Big\}
+\frac{1}{t}\log2,\quad P-a.s.\nonumber
\end{align*}
Hence, we will divide the upcoming proof into two parts.

\textbf{Step 1: we want to verify that there exists a constant $C>0$ such that the following fact holds:  for any $x,y\in \R^d$, there exists a full $P$-measurable set $\widehat{B}(x,y)$ such that
$$\lambda_1(x,y,\omega):=\limsup_{t\rightarrow+\infty}\frac{1}{t}\log|\partial_{x}\Phi_{0,t}^{\delta_y}(\omega,x)v|\leq C$$
for any $v\in \R^d$, $\omega\in \widehat{B}(x,y)$.}

For given $x,y,v\in \R^d$ and $\omega\in B(x,y)$. The properties \eqref{flow} and \eqref{12} imply that
\begin{align*}
&\limsup_{n\rightarrow+\infty}\frac{1}{n}\log|\partial_x\Phi_{0,n}^{\delta_y}(\omega,x)v|\\
&\leq\limsup_{t\rightarrow+\infty}\frac{1}{t}\log|\partial_x\Phi_{0,t}^{\delta_y}(\omega,x)v|\\
&=\lim_{t_n\rightarrow+\infty}\frac{1}{t_n}\log|\partial_x\Phi_{0,t_n}^{\delta_y}(\omega,x)v|\\
&=\lim_{t_n\rightarrow+\infty}\frac{1}{t_n}\log\|\partial_x\Phi_{m_n,t_n}^{P_{m_n}^*\delta_y}(\omega,\Phi_{0,m_n}^{\delta_y}(\omega,x))
\partial_x\Phi_{0,m_n}^{\delta_y}(\omega,x)v|\\
&\leq\lim_{m_n\rightarrow+\infty}\frac{1}{m_n}\log\|\partial_x\Phi_{m_n,t_n}^{P_{m_n}^*\delta_y}(\omega,\Phi_{0,m_n}^{\delta_y}(\omega,x))\|
|\partial_x\Phi_{0,m_n}^{\delta_y}(\omega,x)v|\\
&\leq\lim_{m_n\rightarrow+\infty}\frac{1}{m_n}\log\|\partial_x\Phi_{m_n,t_n}^{P_{m_n}^*\delta_y}(\omega,\Phi_{0,m_n}^{\delta_y}(\omega,x))\|
+\lim_{m_n\rightarrow+\infty}\frac{1}{m_n}\log|\partial_x\Phi_{0,m_n}^{\delta_y}(\omega,x)v|\\
&\leq\lim_{m_n\rightarrow+\infty}\frac{1}{m_n}\log\sup_{t\in[m_n,m_n+1]}\|\partial_x\Phi_{m_n,t}^{P_{m_n}^*\delta_y}(\omega,\Phi_{0,m_n}^{\delta_y}(\omega,x))\|
+\lim_{m_n\rightarrow+\infty}\frac{1}{m_n}\log|\partial_x\Phi_{0,m_n}^{\delta_y}(\omega,x)v|\\
&\leq\lim_{m_n\rightarrow+\infty}\frac{1}{m_n}\log(m_n+d^2e^K)+\lim_{m_n\rightarrow+\infty}\frac{1}{m_n}\log|\partial_x\Phi_{0,m_n}^{\delta_y}(\omega,x)v|\\
&\leq\limsup_{n\rightarrow+\infty}\frac{1}{n}\log|\partial_x\Phi_{0,n}^{\delta_y}(\omega,x)v|,
\end{align*}
where $t_n$ is the time series that ensures the existence of the limit and $m_n$ is the greatest integer less than or equal to $t_n$.
Therefore,  it holds that
$$\limsup_{n\rightarrow+\infty}\frac{1}{n}\log|\partial_{x}\Phi_{0,n}^{\delta_y}(\omega,x)v|
=\limsup_{t\rightarrow+\infty}\frac{1}{t}\log|\partial_{x}\Phi_{0,t}^{\delta_y}(\omega,x)v|.$$

Next, we want to show that
$$\limsup_{n\rightarrow+\infty}\frac{1}{n}\log|\partial_{x}\Phi_{0,n}^{\delta_y}(\omega,x)v|$$
is bounded with respect to $\omega,x,v$. Set
$$\xi_j(\omega):=\|\partial_{x}\Phi_{j,j+1}^{P^*_j\delta_y}(\omega,\Phi_{0,j}^{\delta_y}(\omega,x))\|.$$
It is sufficient to show that $\xi_j$ has finite moments of order $4$. In fact, according to \eqref{DE1}, $\partial_{x}\Phi_{j,j+1}^{P^*_j\delta_y}(z)$ satisfies
$$\partial_{x}\Phi_{j,j+1}^{P^*_j\delta_y}(z)=I+\sum_{k=0}^{d'}\int_j^{j+1}\partial_xV_k(\Phi_{j,r}^{P^*_j\delta_y}(z),P_{X_{r}^y})\partial_{x}\Phi_{j,r}^{P^*_j\delta_y}(z)\,\rmd W^k_r.$$
Through simple calculations, we obtain that for any $p\geq 2$, there exists a constant $C(p,K)$ such that for any $z\in\R^d$ and $j\in\N$,
\begin{align*}
&E[\|\partial_{x}\Phi_{j,j+1}^{P^*_j\delta_y}(z)\|^p]\leq C(p,K).
\end{align*}
Furthermore, for any $j\in\N$, it holds that
\begin{align*}
E\xi_j^p&=E\|\partial_{x}\Phi_{j,j+1}^{P^*_j\delta_y}(\Phi_{0,j}^{\delta_y}(x))\|^p\\
&\leq E\Big[E\bigr[\|\partial_{x}\Phi_{j,j+1}^{P^*_j\delta_y}(z)\|^p\bigr]\bigr|_{z=\Phi_{0,j}^{\delta_y}(x)}\Big]\\
&\leq C(p,K).
\end{align*}
This implies that $E\big(\log^+\xi_j\big)^p\leq C(p,K)$ for any $j\in\N$. Meanwhile, $E\big|\log^+\xi_j-E\log^+\xi_j\big|^4$ can be controlled by a constant independent of $j$.
Then it follows from Cantelli' strong law of large numbers that for any $\omega\in B(x,y)\setminus N(x,y)=:\widehat{B}(x,y)$,
$$\lim_{n\rightarrow+\infty}\frac{1}{n}\sum_{j=0}^{n-1}\Big(\log^+\xi_j(\omega)-E\log^+\xi_j\Big)=0,$$
where $N(x,y)$ is the null set generated by the strong law of large numbers.
Therefore, by property~\eqref{flow}, we obtain that  for any $x,y\in\R^d$, $v\in\R^d$ and $\omega\in \widehat{B}(x,y)$,
\begin{align*}
&\limsup_{n\rightarrow+\infty}\frac{1}{n}\log|\partial_{x}\Phi_{0,n}^{\delta_y}(\omega,x)v|\\
&\leq \limsup_{n\rightarrow+\infty}\frac{1}{n}\log\|\partial_{x}\Phi_{0,n}^{\delta_y}(\omega,x)\|\\
&\leq \limsup_{n\rightarrow+\infty}\frac{1}{n}\sum_{j=0}^{n-1}\log\|\partial_{x}\Phi_{j,j+1}^{P^*_j\delta_y}(\omega,\Phi_{0,j}^{\delta_y}(\omega,x))\|\\
&\leq \limsup_{n\rightarrow+\infty}\frac{1}{n}\sum_{j=0}^{n-1}\log^+\xi_j(\omega)\\
&\leq\limsup_{n\rightarrow+\infty}\frac{1}{n}\sum_{j=0}^{n-1}E\log^+\xi_j
+\limsup_{n\rightarrow+\infty}\frac{1}{n}\sum_{j=0}^{n-1}\Big(\log^+\xi_j(\omega)-E\log^+\xi_j\Big)\\
&\leq C(p,K)=:C.
\end{align*}

\textbf{Step 2: we want to demonstrate that there exists a constant $C'>0$ such that the following fact holds: for any $x,y\in\R^d$, there exists a full $P$-measurable set $\widetilde{B}(y,x)$ such that
 $$\lambda_2(y,x,\omega):=\limsup_{t\rightarrow+\infty}\frac{1}{t}\log|\partial^{x}\Phi_{0,t}^{\delta_x}(\omega,y)v|\leq C'$$
for any $v\in\R^d$ and any $\omega\in \widetilde{B}(y,x)$.}

We first undertake the proof for discrete time. It follows from Lemma \ref{lemma3.4} and the flow property $\Phi_{0,t}^{\delta_x}(y)=\Phi_{r,t}^{P^*_r\delta_x}\big(\Phi_{0,r}^{\delta_x}(y)\big)$ with $0\leq r\leq t$ that
\begin{align*}
&\partial^{x}\Phi_{0,n}^{\delta_x}(\omega,y)\\
&=\partial^{x}\Phi_{n-1,n}^{P_{n-1}^*\delta_x}\big(\omega,\Phi_{0,n-1}^{\delta_x}(\omega,y)\big)\\
&+\partial_{x}\Phi_{n-1,n}^{P_{n-1}^*\delta_x}\big(\omega,\Phi_{0,n-1}^{\delta_x}(\omega,y)\big)\,
\partial^{x}\Phi_{n-2,n-1}^{P_{n-2}^*\delta_x}\big(\omega,\Phi_{0,n-2}^{\delta_x}(\omega,y)\big)\\
&+...\\
&+\partial_{x}\Phi_{n-1,n}^{P_{n-1}^*\delta_x}\big(\omega,\Phi_{0,n-1}^{\delta_x}(\omega,y)\big)\,\partial_{x}\Phi_{n-2,n-1}^{P_{n-2}^*\delta_x}\,
\big(\omega,\Phi_{0,n-2}^{\delta_x}(\omega,y)\big)
\,...\,\partial_{x}\Phi_{1,2}^{P_{1}^*\delta_x}\big(\omega,\Phi_{0,1}^{\delta_x}(\omega,y)\big)\,\partial^{x}\Phi_{0,1}^{\delta_x}(\omega,y).
\end{align*}
This implies that
\begin{align*}
&\limsup_{n\rightarrow+\infty}\frac{1}{n}\log|\partial^{x}\Phi_{0,n}^{\delta_x}(\omega,y)v|\\
&\leq\limsup_{n\rightarrow+\infty}\frac{1}{n}\max\Big\{\log\|\partial^{x}\Phi_{n-1,n}^{P_{n-1}^*\delta_x}\big(\omega,\Phi_{0,n-1}^{\delta_x}(\omega,y)\big)\|,\\
&\qquad\qquad\qquad\qquad\log\|\partial_{x}\Phi_{n-1,n}^{P_{n-1}^*\delta_x}\big(\omega,\Phi_{0,n-1}^{\delta_x}(\omega,y)\big)\,
\partial^{x}\Phi_{n-2,n-1}^{P_{n-2}^*\delta_x}\big(\omega,\Phi_{0,n-2}^{\delta_x}(\omega,y)\big)\|,\\
&\qquad\qquad\qquad\qquad...,\\
&\qquad\qquad\qquad\qquad\log\|\partial_{x}\Phi_{n-1,n}^{P_{n-1}^*\delta_x}\big(\omega,\Phi_{0,n-1}^{\delta_x}(\omega,y)\big)\,\partial_{x}\Phi_{n-2,n-1}^{P_{n-2}^*\delta_x}\,
\big(\omega,\Phi_{0,n-2}^{\delta_x}(\omega,y)\big)\\
&\qquad\qquad\qquad\qquad\qquad\qquad\qquad\qquad\qquad...\,\partial_{x}\Phi_{1,2}^{P_{1}^*\delta_x}\big(\omega,\Phi_{0,1}^{\delta_x}(\omega,y)\big)\,\partial^{x}\Phi_{0,1}^{\delta_x}(\omega,y)\|\Big\}\\
&\leq\limsup_{n\rightarrow+\infty}\frac{1}{n}\max\Big\{\log^+\|\partial^{x}\Phi_{n-1,n}^{P_{n-1}^*\delta_x}\big(\omega,\Phi_{0,n-1}^{\delta_x}(\omega,y)\big)\|,\\
&\qquad\qquad\qquad\qquad\log^+\|\partial_{x}\Phi_{n-1,n}^{P_{n-1}^*\delta_x}(\omega,\Phi_{0,n-1}^{\delta_x}(\omega,y))\|+
\log\|\partial^{x}\Phi_{n-2,n-1}^{P_{n-2}^*\delta_x}(\omega,\Phi_{0,n-2}^{\delta_x}(\omega,y))\|,\\
&\qquad\qquad\qquad\qquad...,\\
&\qquad\qquad\qquad\qquad\sum_{k=2}^{n-1}\log^+\|\partial_{x}\Phi_{k,k+1}^{P_{k}^*\delta_x}(\omega,\Phi_{0,k}^{\delta_x}(\omega,y))\|
+\log\|\partial^{x}\Phi_{1,2}^{P_{1}^*\delta_x}(\omega,\Phi_{0,1}^{\delta_x}(\omega,y))\|,\\
&\qquad\qquad\qquad\qquad\sum_{k=1}^{n-1}\log^+\|\partial_{x}\Phi_{k,k+1}^{P_{k}^*\delta_x}(\omega,\Phi_{0,k}^{\delta_x}(\omega,y))\|
+\log\|\partial^{x}\Phi_{0,1}^{\delta_x}(\omega,y)\|\Big\}\\
&\leq\limsup_{n\rightarrow+\infty}\frac{1}{n}\sum_{k=0}^{n-1}\log^+\|\partial_{x}\Phi_{k,k+1}^{P_{k}^*\delta_x}(\omega,\Phi_{0,k}^{\delta_x}(\omega,y))\|\\
&\qquad\qquad\qquad\qquad\qquad\qquad\qquad+\limsup_{n\rightarrow+\infty}\frac{1}{n}\log\Bigr(\sum_{k=0}^{n-1}\|\partial^{x}\Phi_{k,k+1}^{P_{k}^*\delta_x}(\omega,\Phi_{0,k}^{\delta_x}(\omega,y))\|\Bigr).
\end{align*}
For the first term in the above expression, through the proof in the first part, we obtain that there exists a constant $C_1>0$ such that
\begin{align*}
&\limsup_{n\rightarrow+\infty}\frac{1}{n}\sum_{k=0}^{n-1}\log^+\|\partial_{x}\Phi_{k,k+1}^{P_{k}^*\delta_x}(\omega,\Phi_{0,k}^{\delta_x}(\omega,y))\|\leq C_1.
\end{align*}
For the second term,
\begin{align*}
&\limsup_{n\rightarrow+\infty}\frac{1}{n}\log\Bigr(\sum_{k=0}^{n-1}\|\partial^{x}\Phi_{k,k+1}^{P_{k}^*\delta_x}(\omega,\Phi_{0,k}^{\delta_x}(\omega,y))\|\Bigr)\\
&\leq\limsup_{n\rightarrow+\infty}\frac{1}{n}\log^+\Bigr(\sum_{k=0}^{n-1}E\|\partial^{x}\Phi_{k,k+1}^{P_{k}^*\delta_x}(\Phi_{0,k}^{\delta_x}(y))\|\Bigr)\\
&\qquad\qquad+\limsup_{n\rightarrow+\infty}\frac{1}{n}\log\Bigr|\sum_{k=0}^{n-1}\|\partial^{x}\Phi_{k,k+1}^{P_{k}^*\delta_x}(\omega,\Phi_{0,k}^{\delta_x}(\omega,y))\|
-\sum_{k=0}^{n-1}E\|\partial^{x}\Phi_{k,k+1}^{P_{k}^*\delta_x}(\Phi_{0,k}^{\delta_x}(y))\|\Bigr|.
\end{align*}
Due to Lemma \ref{lem3.6}, the first term on the right-hand side of the inequality can be controlled by a constant, defined by $C_2$. We now deal with the second term. Set $$\xi'_k(\omega):=\|\partial^{x}\Phi_{k,k+1}^{P_{k}^*\delta_x}(\omega,\Phi_{0,k}^{\delta_x}(\omega,y))\|.$$
We aim to show that
\begin{align*}
\limsup_{n\rightarrow+\infty}\frac{1}{n}\log\Bigr|\sum_{k=0}^{n-1}\xi'_k-\sum_{k=0}^{n-1}E\xi'_k\Bigr|\leq C^*,\quad a.s.,
\end{align*}
where $C^*$ is defined in Lemma \ref{lem3.6}.

Note that
\begin{align}\label{SLL}
\lim_{n\rightarrow+\infty}\frac{1}{e^{C^*n}}\Big|\sum_{k=0}^{n-1}\bigr(\xi'_k-E\xi'_k\bigr)\Big|=0, \quad a.s.
\end{align}
is equivalent to that
\begin{align}\label{limt}
\lim_{m\rightarrow+\infty} P\Big\{\bigcup_{n=m}^{+\infty}\Big(\frac{1}{e^{C^*n}}\Big|\sum_{k=0}^{n-1}\bigr(\xi'_k-E\xi'_k\bigr)\Big|>\varepsilon\Big)\Big\}=0
\end{align}
for any $\varepsilon>0$. Due to
\begin{align*}
\bigcup_{n=m}^{+\infty}\Big\{\frac{1}{e^{C^*n}}\Big|\sum_{k=0}^{n-1}\bigr(\xi'_k-E\xi'_k\bigr)\Big|>\varepsilon\Big\}
&\subset \Big\{\sup_{n\geq m}\frac{1}{e^{C^*n}}\Big|\sum_{k=0}^{n-1}\bigr(\xi'_k-E\xi'_k\bigr)\Big|>\varepsilon\Big\},
\end{align*}
it suffices to show that
\begin{align*}
\lim_{m\rightarrow+\infty} P\Big\{\sup_{n\geq m}\frac{1}{e^{C^*n}}\Big|\sum_{k=0}^{n-1}\bigr(\xi'_k-E\xi'_k\bigr)\Big)\Big|>\varepsilon\Big\}=0.
\end{align*}
According to the H\'ajek-R\'enyi inequality,
\begin{align*}
&P\Big\{\max_{m\leq n\leq l}\Big|\frac{1}{e^{C^*n}}\Big(\sum_{k=1}^{n}\bigr(\xi'_{k-1}-E\xi'_{k-1}\bigr)\Big)\Big|>\varepsilon\Big\}\\
&\leq\frac{1}{\varepsilon^2}\Big(\frac{1}{e^{2C^*m}}\sum_{k=1}^{m}E\Big|(\xi'_{k-1}-E\xi'_{k-1})\Big|^2
+\sum_{k=m+1}^{l}\frac{1}{e^{2C^*k}}E\Big|(\xi'_{k-1}-E\xi'_{k-1})\Big|^2\Big)\\
&\leq\frac{2}{\varepsilon^2}\Big(\frac{1}{e^{2C^*m}}\sum_{k=1}^{m}E|\xi'_{k-1}|^2
+\sum_{k=m+1}^{l}\frac{1}{e^{2C^*k}}E|\xi'_{k-1}|^2\Big)\\
&\leq\frac{C}{\varepsilon^2}\Big(\frac{1}{e^{2C^*m}}\sum_{k=1}^{m}e^{C^*k}
+\sum_{k=m+1}^{l}\frac{1}{e^{2C^*k}}e^{C^*k}\Big)\\
&\leq\frac{C}{\varepsilon^2}\Big(\frac{1}{e^{2C^*m}}\sum_{k=1}^{m}e^{C^*k}
+\sum_{k=m+1}^{l}\frac{1}{e^{C^*k}}\Big).
\end{align*}
Due to the continuity of $P$,
\begin{align*}
&P\Big\{\sup_{n\geq m}\Big|\frac{1}{e^n}\Big(\sum_{k=1}^{n}\bigr(\xi'_{k-1}-E\xi'_{k-1}\bigr)\Big)\Big|>\varepsilon\Big\}\\
&\leq\frac{C}{\varepsilon^2}\Big(\frac{1}{e^{2C^*m}}\sum_{k=1}^{m}e^{C^*k}
+\sum_{k=m+1}^{+\infty}\frac{1}{e^{C^*k}}\Big).
\end{align*}
As $m\rightarrow+\infty$, the limit of the right side of the above is zero. This implies that \eqref{SLL} holds. Therefore,
$$\lim_{n\rightarrow+\infty}\Bigr[\frac{1}{n}\log\Bigr|\sum_{k=0}^{n-1}\xi'_k-\sum_{k=0}^{n-1}E\xi'_k\Bigr|-C^*\Bigr]n=-\infty, \quad a.s.$$
This implies
\begin{align*}
\limsup_{n\rightarrow+\infty}\frac{1}{n}\log\Bigr|\sum_{k=0}^{n-1}\xi'_k-\sum_{k=0}^{n-1}E\xi'_k\Bigr|<C^*, \quad a.s.
\end{align*}
We denote the null set as $N'(y,x)$. Then there exists a constant $C'>0$ such that for any $x,y\in\R^d$ and any $\omega\in\Omega\setminus N'(y,x)$,
\begin{align*}
\limsup_{n\rightarrow+\infty}\frac{1}{n}\log\|\partial^{x}\Phi_{0,n}^{\delta_x}(\omega,y)\|\leq C'.
\end{align*}

Finally, we show that the above still holds for continuous time. The properties \eqref{flow} and \eqref{12} imply that for any $x,y\in \R^d$ and $\omega\in B(y,x)$,
\begin{align*}
&\limsup_{t\rightarrow+\infty}\frac{1}{t}\log|\partial^{x}\Phi_{0,t}^{\delta_x}(\omega,y)v|\\
&=\lim_{t\rightarrow+\infty}\frac{1}{t_n}\log|\partial^{x}\Phi_{0,t_n}^{\delta_x}(\omega,y)v|\\
&\leq\limsup_{t\rightarrow+\infty}\frac{1}{m_n}\log\|\partial^{x}\Phi_{m_n,t_n}^{P_{m_n}^*\delta_x}(\omega,\Phi_{0,m_n}^{\delta_x}(\omega,y))
+\partial_{x}\Phi_{m_n,t_n}^{P_{m_n}^*\delta_x}(\omega,\Phi_{0,m_n}^{\delta_x}(\omega,y))\partial^{x}\Phi_{0,m_n}^{\delta_x}(\omega,y)\|\\
&\leq\limsup_{m_n\rightarrow+\infty}\max\Big\{\frac{1}{m_n}\log\|\partial^{x}\Phi_{m_n,t_n}^{P_{m_n}^*\delta_x}(\omega,\Phi_{0,m_n}^{\delta_x}(\omega,y))\|,\\
&\qquad\qquad\qquad\qquad\qquad\qquad\qquad\qquad\qquad\frac{1}{m_n}\log\|\partial_{x}\Phi_{m_n,t_n}^{P_{m_n}^*\delta_x}(\omega,\Phi_{0,m_n}^{\delta_x}(\omega,y))\partial^{x}\Phi_{0,m_n}^{\delta_x}(\omega,y)\|\Big\}\\
&\leq\max\Big\{\limsup_{m_n\rightarrow+\infty}\frac{1}{m_n}\log\|\partial^{x}\Phi_{m_n,t_n}^{P_{m_n}^*\delta_x}(\omega,\Phi_{0,m_n}^{\delta_x}(\omega,y))\|,\\
&\qquad\qquad\qquad\qquad\qquad\qquad\qquad\qquad\qquad\limsup_{m_n\rightarrow+\infty}\frac{1}{m_n}\log(m_n+dK^2)\|\partial^{x}\Phi_{0,m_n}^{\delta_x}(\omega,y)\|\Big\}\\
&\leq\max\Big\{\limsup_{m_n\rightarrow+\infty}\frac{1}{m_n}\log\|\partial^{x}\Phi_{m_n,t_n}^{P_{m_n}^*\delta_x}(\omega,\Phi_{0,m_n}^{\delta_x}(\omega,y))\|
,\limsup_{m_n\rightarrow+\infty}\frac{1}{m_n}\log\|\partial^{x}\Phi_{0,m_n}^{\delta_x}(\omega,y)\|\Big\}\\
&\leq\max\Big\{\limsup_{m_n\rightarrow+\infty}\frac{1}{m_n}\log\|\partial^{x}\Phi_{m_n,t_n}^{P_{m_n}^*\delta_x}(\omega,\Phi_{0,m_n}^{\delta_x}(\omega,y))\|
,C'\Big\},
\end{align*}
where $t_n$ is the time series that ensures the existence of the limit and  $m_n$ is the greatest integer less than or equal to $t_n$.
Set $$\overline{\xi}_k(\omega):=\sup_{t\in[k,k+1)}\|\partial^{x}\Phi_{k,t}^{P_{k}^*\delta_x}(\omega,\Phi_{0,k}^{\delta_x}(\omega,y))\|.$$
We obtain that
\begin{align*}
&\limsup_{m_n\rightarrow+\infty}\frac{1}{m_n}\log\|\partial^{x}\Phi_{m_n,t_n}^{P_{m_n}^*\delta_x}(\omega,\Phi_{0,m_n}^{\delta_x}(\omega,y))\|\\
&\leq \limsup_{m_n\rightarrow+\infty}\frac{1}{m_n}\log\sup_{t\in[m_n,m_n+1)}\|\partial^{x}\Phi_{m_n,t}^{P_{m_n}^*\delta_x}(\omega,\Phi_{0,m_n}^{\delta_x}(\omega,y))\|\\
&\leq \limsup_{m_n\rightarrow+\infty}\frac{1}{m_n}\log\sum_{k=0}^{m_n}\sup_{t\in[k,k+1)}\|\partial^{x}\Phi_{m_n,t}^{P_{m_n}^*\delta_x}(\omega,\Phi_{0,m_n}^{\delta_x}(\omega,y))\|\\
&\leq \limsup_{m_n\rightarrow+\infty}\frac{1}{m_n}\log\sum_{k=0}^{m_n}\overline{\xi}_k(\omega)\\
&\leq \limsup_{n\rightarrow+\infty}\frac{1}{n}\log\sum_{k=0}^{n-1}E\overline{\xi}_k+ \limsup_{n\rightarrow+\infty}\frac{1}{n}\log\sum_{k=0}^{n-1}\Big(\overline{\xi}_k(\omega)-E\overline{\xi}_k\Big).
\end{align*}
The remaining proof is similar to that of the discrete case. Ultimately, for any $x,y\in \R^d$, we establish the upper bound for $\lambda_2(y,x,\cdot)$ on $B(y,x)\setminus N'(y,x)=:\widetilde{B}(y,x)$. We still denote this upper bound as $C'$, which is independent of $x,y$.

Let $y=x$ and $\widehat{B}(x,x)\bigcap \widetilde{B}(x,x)=:\Omega(x)$, which is a full $P$-measurable set. Finally, we obtain that for any $x,v\in\R^d$, $\omega\in\Omega(x)$, the following formula holds:
$$\limsup_{t\rightarrow+\infty}\frac{1}{t}\log|\partial_x\Phi_{0,t}(\omega,x)v|\leq \max\{C,C'\}=:\kappa.$$
The proof of Theorem \ref{Boundedness} is complete.
\end{proof}

\begin{rem}
The proof in the second part differs from the first part due to distributional shifts, making it impossible to control $E\bigr\|\partial^{x}\Phi_{n-1,n}^{P_{n-1}^*\delta_x}(\Phi_{n-1}^{\delta_x}(y))\bigr\|^p$ by a constant independent of $n$, as indicated in Lemma~\ref{lem3.6}. This prevents us from directly applying the strong law of large numbers. We need to perform a transformation \eqref{SLL} and then apply the methods used to prove the strong law of large numbers to establish the upper bound.
\end{rem}

\section{An Example}\label{section4}

In this section, we present a specific McKean-Vlasov SDE with coefficients that are sufficiently smooth, and the first derivatives are bounded. However, we cannot define the Lyapunov exponent using the limit in this example as the limit does not exist. Simultaneously, this example illustrates the reasonability for employing the upper limit to define Lyapunov exponents in Theorem~\ref{MET1}. It is worth noting that under the same conditions, for classical SDEs, the Lyapunov exponent can be defined using a limit, as the limit does exist. This indicates significant differences in dynamics between classical SDEs and McKean-Vlasov SDEs. Furthermore, this example further reveals the significant impact of distributions on the system.

Consider the following continuous function
\begin{equation}
g(x)=
\left\{
\begin{aligned}
\nonumber
&x-\frac{1}{2},~&&x\in(-\infty,\frac{1}{2}),\\
&0,~&&x\in[\frac{1}{2},a_0),\\
&...\\
&\frac{1}{4}x-\frac{1}{4}a_{2n-2}+\frac{1}{4}a_{2n-3}+...+\frac{1}{4}a_1-\frac{1}{4}a_0,~&&x\in[a_{2n-2},a_{2n-1}),\\
&\frac{1}{4}a_{2n-1}-\frac{1}{4}a_{2n-2}+\frac{1}{4}a_{2n-3}+...+\frac{1}{4}a_1-\frac{1}{4}a_0,~&&x\in[a_{2n-1},a_{2n}),\\
&...,
\end{aligned}
\right.
\end{equation}
where the sequence $\{a_n\}\subseteq\R^+~(a_0>1)$ is strictly monotonically increasing sequence. Next, we determine the values of $\{a_n\}$. Consider the ODE
\begin{align*}
Y_t=1+\int_0^tY_{r}-g(Y_{r})\,\rmd r.
\end{align*}
Regardless of the values of the sequence $\{a_n\}$, the solution $Y_{t}~(t\geq0)$ exists uniquely and is strictly monotonically increasing with respect to $t$. We define $Y_{S_n}$ as $a_n$ successively, where $S_n =2^n$.

We polish $g$ such that the function $\overline{g}$ after polishing is infinitely differentiable, and satisfies
\begin{equation*}
\left\{
\begin{aligned}
&0\leq \overline{g}'(x)\leq\frac{1}{4}, &&x\in[1,+\infty), \\
&\overline{g}(x)=g(x), &&x\in(a_n+\varepsilon_n,a_{n+1}-\varepsilon_{n+1}),
\end{aligned}
\right.
\end{equation*}
where $\varepsilon_n=e^{-S_n^2}$. Consider the ODE
\begin{align}\label{ODE}
Y_t=1+\int_0^tY_{r}-\overline{g}(Y_{r})\,\rmd r.
\end{align}
The solution to this equation always exists uniquely, denoted by $\overline{Y}_{t}~(t\geq0)$. It should be noted that $\overline{Y}_{\cdot}$ is strictly monotonically increasing. We denote $\overline{g}'(\overline{Y}_t)$ by $f(t)$ and choose three suitable positive sequences $T_n,\delta_n,\delta'_n$ such that for any $n\geq0$, $\overline{Y}_{T_n-\delta_{n}}=a_n-\varepsilon_n$, $\overline{Y}_{T_n+\delta'_{n}}=a_n+\varepsilon_n$, $\overline{Y}_{T_n}=a_n$. It should be noted that $\delta_n,\delta'_n=o(e^{-S_n^2})$ because of $\varepsilon_n=e^{-S_n^2}$.

Moreover, the expression of $g$ implies that for any $n\geq 1$,
\begin{equation}\label{f}
f(t)=
\left\{
\begin{aligned}
&\leq \frac{1}{4}, && t\in[T_{2n-2}-\delta_{2n-2},T_{2n-2}+\delta'_{2n-2}], \\
&\frac{1}{4}, && t\in(T_{2n-2}+\delta'_{2n-2},T_{2n-1}-\delta_{2n-1}), \\
&\leq \frac{1}{4}, && t\in[T_{2n-1}-\delta_{2n-1},T_{2n-1}+\delta'_{2n-1}], \\
&0, && t\in(T_{2n-1}+\delta'_{2n-1},T_{2n}-\delta_{2n}).
\end{aligned}
\right.
\end{equation}
This achieves alternating occurrences of rapid and slow growth for the particle distribution on a time partition with exponentially increasing interval lengths.
Set $c_n:=T_n-S_n$. Then it holds that
\begin{align}\label{ST1}
c:=\sum_{l=0}^{+\infty}(c_{2l+2}-c_{2l})<+\infty,
\end{align}
and
\begin{align}\label{ST}
\lim_{n\rightarrow+\infty}\frac{c_{n+l}}{T_n}=0,\quad \text{for all}~ l>0.
\end{align}

We are now ready to state and prove the main result of this section. Consider the following $1$-dimensional McKean-Vlasov SDE
\begin{equation}\label{example}
\left\{
\begin{aligned}
\rmd X_t&=\bigr(X_t-\overline{g}(EX_t)\bigr)\,\rmd t+X_t\,\rmd W_t,\\
X_0&=x,\\
\end{aligned}
\right.
\end{equation}
where $W_t$ is a $1$-dimensional Brownian motion. The coefficients of the equation belong to $\cal{C}_{b,Lip}^{n.n}(\R\times\cal{P}_2(\R);\R)(n\geq1)$. Hence, \cite[Theorem 3.3]{CL} implies that the unique solution of the equation exists globally and there exists a modification of this solution, denoted by $\Phi_{0,t}(x)$,  such that for all $t\geq0$, the mapping $x\mapsto\Phi_{0,t}(x)$ is $P$-$a.s.$ differentiable.

Set $F(P_v):=\overline{g}(Ev)$, $P_v\in\cal{P}_2(\R)$. Then $F$ is Lions differentiable in $\cal{P}_2(\R)$ and the Lions derivative of $F$ at $P_{v}\in\cal{P}_2(\R)$ satisfies $\partial_\mu F(P_{v},y)=\overline{g}'(Ev)$ for any $y\in\R$; see \cite[Example~2.2]{BLPR} for details. Therefore, the Jacobian matrix of $\Phi_{0,t}(x)$ satisfies
\begin{align*}
\partial_x\Phi_{0,t}(x)=1+\int_0^t\partial_x\Phi_{0,r}(x)-\overline{g}'(E[\Phi_{0,r}(x)])E\big[\partial_x\Phi_{0,r}(x)\big]\,\rmd r+\int_0^t\partial_x\Phi_{0,r}(x)\,\rmd W_r.
\end{align*}
Through straightforward calculations, we obtain
$$E[\partial_x\Phi_{0,t}(x)]=e^{\int_0^t(1-\overline{g}'(E[\Phi_{0,r}(x)]))\,\rmd r}.$$

Let $x=1$. It should be noted that $E[\Phi_{0,t}(1)]$ is the unique solution of ODE \eqref{ODE}. Therefore, $\overline{g}'(E[\Phi_{0,t}(1)])=f(t)$.
By \cite[Theorem 8.4.2]{ARn},
\begin{align*}
\partial_x\Phi_{0,t}(1)&=e^{\frac{1}{2}t+W_t}\Big(1-\int_0^te^{-\frac{1}{2}r-W_r}\,f(r)\,e^{\int_0^r(1-f(u))\rmd u}\rmd r\Big)\\
&=e^{\frac{1}{2}t+W_t}\Big(1-\int_0^tf(r)\,e^{\int_0^r(\frac{1}{2}-f(u))\rmd u-W_r}\rmd r\Big).
\end{align*}
Further, it holds that for sufficiently large $t$ and almost all $\omega$,
\begin{align}\label{Ly}
&\frac{1}{t}\log|\partial_x\Phi_{0,t}(\omega,1)|\\
&=\frac{1}{t}\log e^{\frac{1}{2}t+W_t(\omega)}+\frac{1}{t}\log\Big|1-\int_0^te^{-\frac{1}{2}r-W_r(\omega)}\,f(r)\,e^{\int_0^r(1-f(u))\rmd u}\rmd r\Big|\nonumber\\
&=\frac{1}{2}+\frac{W_t(\omega)}{t}+\frac{1}{t}\log\Big(\int_0^tf(r)\,e^{\int_0^r(\frac{1}{2}-f(u))\rmd u-W_r(\omega)}\rmd r-1\Big)\nonumber\\
&=\frac{1}{2}+\frac{W_t(\omega)}{t}+\frac{1}{t}\log \Big(e^{\int_0^{1-\delta_0}(\frac{1}{2}-f(u))\rmd u}\int_{1-\delta_0}^tf(r)\,e^{\int_{1-\delta_0}^r(\frac{1}{2}-f(u))\rmd u-W_r(\omega)}\rmd r+C(\omega)\Big),\nonumber
\end{align}
where $$C(\omega)=\int_0^{1-\delta_0}f(r)\,e^{\int_0^r(\frac{1}{2}-f(u))\rmd u-W_r(\omega)}\rmd r-1.$$

It is well-known that the Brownian motion satisfies the property:
\begin{align}\label{W}
\lim\limits_{t\rightarrow+\infty}\dfrac{W(t)}{t}=0, \quad a.s.
\end{align}
This implies that for almost all $\omega\in \Omega$, any $\epsilon>0$, there exists $K(\omega)>0$ such that for $r\geq T_{K(\omega)}$,
\begin{align}\label{W0}
-\epsilon\, r\leq -W_r(\omega)\leq\epsilon\, r.
\end{align}
Further, for any $k>\frac{K(\omega)}{2}$, it holds that
\begin{align}\label{W1}
 \int_{T_{2k}+\delta'_{2k}}^{T_{2k+1}-\delta_{2k+1}}e^{\frac{1}{4}r-W_r(\omega)}\rmd r\leq \frac{4}{1+4\epsilon}\Big(e^{(\frac{1}{4}+\epsilon)(T_{2k+1}-\delta_{2k+1})}-e^{(\frac{1}{4}+\epsilon)(T_{2k}+\delta'_{2k})}\Big)
\end{align}
and
\begin{align}\label{W2}
\int_{T_{2k}+\delta'_{2k}}^{T_{2k+1}-\delta_{2k+1}}e^{\frac{1}{4}r-W_r(\omega)}\rmd r\geq\frac{4}{1-4\epsilon}\Big(e^{(\frac{1}{4}-\epsilon)(T_{2k+1}-\delta_{2k+1})}-e^{(\frac{1}{4}-\epsilon)(T_{2k}+\delta'_{2k})}\Big).
\end{align}

We set
$$A_{n}:=\int_{1-\delta_0}^{T_{n}}f(r)\,e^{\int_{1-\delta_0}^r(\frac{1}{2}-f(u))\rmd u-W_r}\rmd r,~n\geq1.$$
Choose two subsequences $\{T_{2n}\}$ and $\{T_{2n-1}\}$. It follows from \eqref{Ly} and \eqref{W} that
\begin{align*}
&\limsup_{n\rightarrow+\infty}\frac{1}{T_{2n}}\log|\partial_x\Phi_{0,T_{2n}}(\omega,1)|=\frac{1}{2}+\limsup_{n\rightarrow+\infty}\frac{1}{T_{2n}}\log A_{2n}(\omega)
\end{align*}
and
\begin{align*}
&\liminf_{n\rightarrow+\infty}\frac{1}{T_{2n-1}}\log|\partial_x\Phi_{0,T_{2n-1}}(\omega,1)|
=\frac{1}{2}+\liminf_{n\rightarrow+\infty}\frac{1}{T_{2n-1}}\log A_{2n-1}(\omega).
\end{align*}

\begin{lemma}\label{noexist}
For almost all $\omega$,
$$\limsup_{n\rightarrow+\infty}\frac{1}{T_{2n}}\log|\partial_x\Phi_{0,T_{2n}}(\omega,1)|
<\liminf_{n\rightarrow+\infty}\frac{1}{T_{2n-1}}\log|\partial_x\Phi_{0,T_{2n-1}}(\omega,1)|.$$
\end{lemma}

\begin{proof}
It follows from \eqref{f}, \eqref{ST1} and \eqref{W0} that for sufficiently large $n$,
\begin{align*}
&A_{2n}(\omega)\\
&=\int_{1-\delta_0}^{T_{2n}}f(r)\,e^{\int_{1-\delta_0}^r(\frac{1}{2}-f(u))\rmd u-W_r(\omega)}\rmd r\\
&\leq\Big(\sum_{k=0}^{n-1}\int_{T_{2k}+\delta'_{2k}}^{T_{2k+1}-\delta_{2k+1}}
+\sum_{k=0}^{2n}\int_{T_{k}-\delta_{k}}^{T_{k}+\delta'_{k}}\Big)f(r)\,e^{\int_{1-\delta_0}^r\frac{1}{2}-f(u)\rmd u-W_r(\omega)}\rmd r\\
&\leq\frac{1}{4}\sum_{k=0}^{n-1}e^{\int_{1-\delta_0}^{S_{2k}+c_{2k}+\delta'_{2k}}\frac{1}{2}-f(r)\rmd r}
\int_{T_{2k}+\delta'_{2k}}^{T_{2k+1}-\delta_{2k+1}}e^{\int_{T_{2k}+\delta'_{2k}}^{r}\frac{1}{2}-f(u)\rmd u-W_r(\omega)}\rmd r\\
&\qquad\qquad\qquad\qquad\qquad\qquad+\frac{1}{4}\sum_{k=0}^{2n}\int_{T_{k}-\delta_{k}}^{T_{k}+\delta'_{k}}e^{\int_{1-\delta_0}^r\frac{1}{2}-f(u)\rmd u-W_r(\omega)}\rmd r\\
&\leq C\sum_{k=0}^{n-1}e^{\frac{5}{12}S_{2k}+\frac{1}{2}\sum_{l=0}^{k-1}(c_{2l+2}-c_{2l})}
\int_{T_{2k}+\delta'_{2k}}^{T_{2k+1}-\delta_{2k+1}}e^{\frac{1}{4}r-\frac{1}{4}(T_{2k}+\delta'_{2k})-W_r(\omega)}\rmd r\\
&\qquad\qquad\qquad\qquad\qquad\qquad+C\Big(\sum_{k=0}^{K(\omega)}+\sum_{k=K(\omega)+1}^{2n}\Big)\int_{T_{k}-\delta_{k}}^{T_{k}+\delta'_{k}}e^{\frac{1}{2}r-W_r(\omega)}\rmd r\\
&\leq C'\sum_{k=0}^{n-1}e^{\frac{5}{12}T_{2k}-\frac{1}{4}T_{2k}}\int_{T_{2k}+\delta'_{2k}}^{T_{2k+1}-\delta_{2k+1}}
e^{\frac{1}{4}r-W_r(\omega)}\rmd r
+C\sum_{k=K(\omega)+1}^{2n}\int_{T_{k}-\delta_{k}}^{T_{k}+\delta'_{k}}e^{\frac{1}{2}r-W_r(\omega)}\rmd r+C(\omega)\\
&\leq C'\sum_{k=0}^{n-1}e^{\frac{1}{6}T_{2k}}\int_{T_{2k}+\delta'_{2k}}^{T_{2k+1}-\delta_{2k+1}}e^{\frac{1}{4}r-W_r(\omega)}\rmd r
+C\sum_{k=K(\omega)+1}^{2n}(\delta'_{k}+\delta_{k})e^{(\frac{1}{2}-\epsilon)(T_{k}+\delta'_{k})}+C(\omega)\\
&\leq C'\sum_{k=0}^{n-1}e^{\frac{1}{6}T_{2k}}\int_{T_{2k}+\delta'_{2k}}^{T_{2k+1}-\delta_{2k+1}}e^{\frac{1}{4}r-W_r(\omega)}\rmd r
+C'(\omega),
\end{align*}
where $C'(\omega)$ is finite because of $\delta'_{n}+\delta_{n}=o(e^{-S_n^2})$, $T_n=S_n+c_n$ and \eqref{ST}. Similarly, we obtain
\begin{align*}
A_{2n-1}(\omega)&=\int_{1-\delta}^{T_{2n-1}}f(r)\,e^{\int_{1-\delta_0}^r(\frac{1}{2}-f(u))\rmd u-W_r(\omega)}\rmd r\\
&\geq\frac{1}{4}\sum_{k=0}^{n-1}\int_{T_{2k}+\delta'_{2k}}^{T_{2k+1}-\delta_{2k+1}}e^{\int_{1-\delta_0}^r\frac{1}{2}-f(u)\rmd u-W_r(\omega)}\rmd r\\
&\geq\frac{1}{4}\sum_{k=0}^{n-1}e^{\int_{1-\delta_0}^{T_{2k}+\delta'_{2k}}\frac{1}{2}-f(r)\rmd r}\int_{T_{2k}+\delta'_{2k}}^{T_{2k+1}-\delta_{2k+1}}e^{\frac{1}{4}r-\frac{1}{4}(T_{2k}+\delta'_{2k})-W_r(\omega)}\rmd r\\
&\geq C''\sum_{k=0}^{n-1}e^{\frac{5}{12}T_{2k}-\frac{1}{4}T_{2k}}
\int_{T_{2k}+\delta'_{2k}}^{T_{2k+1}-\delta_{2k+1}}e^{\frac{1}{4}r-W_r(\omega)}\rmd r\\
&=C''\sum_{k=0}^{n-1}e^{\frac{1}{6}T_{2k}}\int_{T_{2k}+\delta'_{2k}}^{T_{2k+1}-\delta_{2k+1}}e^{\frac{1}{4}r-W_r(\omega)}\rmd r.
\end{align*}
Hence, the facts $T_n=S_n+c_n$, \eqref{ST} and \eqref{W1} imply that
\begin{align*}
&\limsup_{n\rightarrow+\infty}\frac{1}{T_{2n}}\log A_{2n}(\omega)\\
&=\limsup_{n\rightarrow+\infty}\frac{1}{T_{2n}}\log \Bigr(C'\Bigr[\sum_{k=0}^{K(\omega)}+\sum_{k=K(\omega)}^{n-1}\Bigr]e^{\frac{1}{6}T_{2k}}\int_{T_{2k}+\delta'_{2k}}^{T_{2k+1}-\delta_{2k+1}}e^{\frac{1}{4}r-W_r(\omega)}\rmd r+C'(\omega)\Bigr)\\
&\leq\limsup_{n\rightarrow+\infty}\frac{1}{T_{2n}}\log (n-K(\omega))\frac{4}{1+4\epsilon}e^{\frac{1}{6}T_{2n-2}}
\Big(e^{(\frac{1}{4}+\epsilon)(T_{2n-1}-\delta_{2n-1})}-e^{(\frac{1}{4}+\epsilon)(T_{2n-2}+\delta'_{2n-2})}\Big)\\
&\leq\limsup_{n\rightarrow+\infty}\frac{1}{T_{2n}}\log e^{\frac{1}{6}T_{2n-2}}
\Big(e^{(\frac{1}{4}+\epsilon)(T_{2n-1}-\delta_{2n-1})}-e^{(\frac{1}{4}+\epsilon)(T_{2n-2}+\delta'_{2n-2})}\Big)\\
&\leq\limsup_{n\rightarrow+\infty}\frac{1}{T_{2n}}\log e^{(\frac{1}{6}+\frac{1}{4}+\epsilon)T_{2n-2}}
\Big(e^{(\frac{1}{4}+\epsilon)(T_{2n-1}-T_{2n-2})-(\frac{1}{4}+\epsilon)\delta_{2n-1}}-e^{(\frac{1}{4}+\epsilon)\delta'_{2n-2}}\Big)\\
&\leq\limsup_{n\rightarrow+\infty}\frac{1}{T_{2n}}\log
e^{(\frac{2}{3}+2\epsilon)T_{2n-2}+(\frac{1}{4}+\epsilon)(c_{2n-1}-2c_{2n-2})-(\frac{1}{4}+\epsilon)\delta_{2n-1}}\\
&\leq\limsup_{n\rightarrow+\infty}\frac{1}{T_{2n}}\log
e^{(\frac{1}{6}+\frac{1}{2}\epsilon)T_{2n}+(\frac{2}{3}+2\epsilon)(c_{2n-2}-\frac{1}{4}c_{2n})+(\frac{1}{4}+\epsilon)(c_{2n-1}-2c_{2n-2})-(\frac{1}{4}+\epsilon)\delta_{2n-1}}\\
&=\frac{1}{2}\epsilon+\frac{1}{6}.
\end{align*}
Moreover,
\begin{align*}
&\liminf_{n\rightarrow+\infty}\frac{1}{T_{2n-1}}\log A_{2n-1}(\omega)\\
&\geq\liminf_{n\rightarrow+\infty}\frac{1}{T_{2n-1}}\log \Bigr(C''\Bigr[\sum_{k=0}^{K(\omega)}+\sum_{k=K(\omega)}^{n-1}\Bigr]e^{\frac{1}{6}T_{2k}}
\int_{T_{2k}+\delta'_{2k}}^{T_{2k+1}-\delta_{2k+1}}e^{\frac{1}{4}r-W_r(\omega)}\rmd r\Bigr)\\
&=\liminf_{n\rightarrow+\infty}\frac{1}{T_{2n-1}}\log \sum_{k=K(\omega)}^{n-1}e^{\frac{1}{6}T_{2k}}\int_{T_{2k}+\delta'_{2k}}^{T_{2k+1}-\delta_{2k+1}}e^{\frac{1}{4}r-W_r(\omega)}\rmd r\\
&\geq\liminf_{n\rightarrow+\infty}\frac{1}{T_{2n-1}}\log\frac{4}{1-4\epsilon} \sum_{k=K(\omega)}^{n-1}\Big(e^{(\frac{1}{4}-\epsilon+\frac{1}{12})T_{2k+1}-(\frac{1}{4}-\epsilon)\delta_{2k+1}-\frac{1}{12}c_{2k+1}+\frac{1}{6}c_{2k}}\\
&\qquad\qquad\qquad\qquad\qquad\qquad\qquad\qquad-e^{(\frac{1}{4}-\epsilon+\frac{1}{6})T_{2k}+(\frac{1}{4}-\epsilon)\delta'_{2k}}\Big)\\
&\geq\liminf_{n\rightarrow+\infty}\frac{1}{T_{2n-1}}\log \Big(e^{(\frac{1}{3}-\epsilon)T_{2n-1}-(\frac{1}{4}-\epsilon)\delta_{2n-1}-\frac{1}{12}c_{2n-1}+\frac{1}{6}c_{2n-2}}
-e^{(\frac{5}{12}-\epsilon)T_{2n-2}+(\frac{1}{4}-\epsilon)\delta'_{2n-2}}\Big)\\
&\geq\liminf_{n\rightarrow+\infty}\frac{1}{T_{2n-1}}\log e^{(\frac{1}{3}-\epsilon)T_{2n-1}}
\Big(e^{-(\frac{1}{4}-\epsilon)\delta_{2n-1}-\frac{1}{12}c_{2n-1}+\frac{1}{6}c_{2n-2}}\\
&\qquad\qquad\qquad\qquad\qquad\qquad\qquad\qquad-e^{-(\frac{1}{8}-\frac{1}{2}\epsilon)T_{2n-1}+(\frac{1}{4}\epsilon+\frac{1}{6})(c_{2n-2}-\frac{1}{2}c_{2n-1})+(\frac{1}{4}-\epsilon)\delta'_{2n-2}}\Big)\\
&=\frac{1}{3}-\epsilon.
\end{align*}
Hence, by \eqref{Ly} and \eqref{W},
\begin{align*}
\limsup_{n\rightarrow+\infty}\frac{1}{T_{2n}}\log|\partial_x\Phi_{0,T_{2n}}(\omega,1)|&=\frac{1}{2}+\limsup_{n\rightarrow+\infty}\frac{1}{T_{2n}}\log A_{2n}(\omega)\\
&\leq \frac{1}{2}+\frac{1}{2}\epsilon+\frac{1}{6}.
\end{align*}
Similarly,
\begin{align*}
&\liminf_{n\rightarrow+\infty}\frac{1}{T_{2n-1}}\log|\partial_x\Phi_{0,T_{2n-1}}(\omega,1)|\geq \frac{1}{2}+\frac{1}{3}-\epsilon.
\end{align*}
We can choose a suitable $\epsilon>0$ satisfying $$\frac{1}{3}-\epsilon>\frac{1}{2}\epsilon+\frac{1}{6},$$
such that for almost all $\omega$,
$$\limsup_{n\rightarrow+\infty}\frac{1}{T_{2n}}\log|\partial_x\Phi_{0,T_{2n}}(\omega,1)|
<\liminf_{n\rightarrow+\infty}\frac{1}{T_{2n-1}}\log|\partial_x\Phi_{0,T_{2n-1}}(\omega,1)|.$$
\end{proof}
\begin{rem}
This lemma implies that for almost all $\omega\in\Omega$, $$\lim_{t\rightarrow+\infty}\frac{1}{t}\log|\partial_x\Phi_{0,t}(\omega,1)|$$ does not exist.
\end{rem}

Above, we considered the case of  $x=1$.  For any $x\leq a_0$ and $x\neq1$, consider the following ODE with the initial value $x\in\R$:
$$\overline{Y}_t(x)=x+\int_0^t\overline{Y}_{r}(x)-\overline{g}(\overline{Y}_{r}(x))\,\rmd r.$$
Since the coefficients of this equation are positive, there exists a constant $\tau>0$ such that $\overline{Y}_{\tau}(x)=a_0$. So in this case, only a finite time shift relative to $x=1$ has occurred.

For any $x>a_0$, there exists a constant $l\in\mathbb{N}$, such that $x\in(a_l,a_{l+1}]$. Without loss of generality, assume that $l$ is even. The proof for the case when $l$ is odd follows similarly.
We set $f(t,x):=\overline{g}'(\overline{Y}_{t}(x))$. Then there exist three subsequences of $\R^+$, denoted as $\{\Gamma_n\}_{n\geq l+1}$, $\{\sigma_n\}_{n\geq l+1}$ and $\{\sigma'_n\}_{n\geq l+1}$, such that for any $k\geq 0$, $\overline{Y}_{\Gamma_{l+k}}(x)=a_{l+k}$, $\overline{Y}_{\Gamma_{l+k}+\sigma'_{l+k}}(x)=a_{l+k}+\varepsilon_{l+k}$, $\overline{Y}_{\Gamma_{l+k}-\sigma_{l+k}}(x)=a_{l+k}-\varepsilon_{l+k}$,
$\Gamma_{l+k+1}-\Gamma_{l+k}=S_{l+k}+o(e^{-S_{l+k}^2})$. Moreover, it holds that $\sigma_k,\sigma'_k=o(e^{-S_k^2})$ because of $\varepsilon_k=e^{-S_k^2}$. For given $x>a_0$, the expression of $g$ implies that
\begin{equation}
f(t,x)=
\left\{
\begin{aligned}
\nonumber
&\leq\frac{1}{4},~&&t\in(0,\Gamma_{l+2}-\sigma_{l+2}),\\
&...\\
&\leq\frac{1}{4},~&&t\in[\Gamma_{l+2n-2}-\sigma_{l+2n-2},\Gamma_{l+2n-2}+\sigma'_{l+2n-2}],\\
&\frac{1}{4},~&&t\in(\Gamma_{l+2n-2}+\sigma'_{l+2n-2},\Gamma_{l+2n-1}-\sigma_{l+2n-1}),\\
&\leq\frac{1}{4},~&&t\in[\Gamma_{l+2n-1}-\sigma_{l+2n-1},\Gamma_{l+2n-1}+\sigma'_{l+2n-1}],\\
&0,~&&t\in(\Gamma_{l+2n-1}+\sigma'_{l+2n-1},\Gamma_{l+2n}-\sigma_{l+2n}),\\
&...
\end{aligned}
\right.
\end{equation}
The remaining proof is similar to the case when $x=1$.

In conclusion, for any $x\in\R$ and almost all $\omega\in\Omega$, we can prove that the limit $$\lim_{t\rightarrow+\infty}\frac{1}{t}\log|\partial_x\Phi_{0,t}(\omega,x)|$$ does not exist.

\begin{rem}
In this example, the coefficients' dependency on the distribution is fully determined by an ODE, transforming the autonomous McKean-Vlasov SDE into a non-autonomous classical SDE. By constructing the coefficients of this ODE, we achieve alternating occurrences of rapid and slow growth for the particle distribution on a time partition with exponentially increasing interval lengths.
\end{rem}

\section*{Acknowledgements}
This work is supported by NSFC Grants 11871132, 11925102, National Key R\&D Program of China (No. 2023YFA1009200), LiaoNing Revitalization Talents Program (Grant XLYC2202042), and Dalian High-level Talent Innovation Program (Grant 2020RD09).

\end{document}